\documentclass[10pt]{article}

\usepackage{amsmath,amsfonts,amssymb,amsthm}

\usepackage{tikz}
\usepackage{wrapfig} 
\usepackage[all]{xy}

\usepackage{shuffle}

\usepackage[colorlinks=true]{hyperref}

\usepackage{color, colortbl}
\definecolor{paleblue}{rgb}{0.7, 0.7, 1.0}
\definecolor{palegreen}{rgb}{0.7, 1,0.7}

  \theoremstyle{plain}
\newtheorem{conj}{Conjecture}
\newtheorem{theorem}{Theorem}
\newtheorem{proposition}{Proposition}[section]
\newtheorem{corollary}[proposition]{Corollary}
  \theoremstyle{remark}
\newtheorem{remark}[proposition]{Remark}
  \theoremstyle{definition}
\newtheorem{definition}[proposition]{Definition}
\newtheorem{lemma}[proposition]{Lemma}
\newtheorem{example}[proposition]{Example}

\newcommand{\wlhd}{\widetilde{\lhd}}
\newcommand{\II}{\mathbf{I}}
\newcommand{\NN}{\mathbb{N}}
\newcommand{\ZZ}{\mathbb{Z}}

\newcommand{\RR}{\mathbb{R}}
\newcommand{\CC}{\mathbb{C}}

\newcommand{\C}{\mathcal C}
\newcommand{\Cbr}{\mathcal{C}_{br}}
\newcommand{\CCbr}{\mathcal{C}_{2br}}
\newcommand{\B}{\mathcal B}
\newcommand{\D}{\mathcal D}
\newcommand{\f}{\mathcal F}
\newcommand{\of}{\overline{\f}}
\renewcommand{\d}{\mathcal D}
\newcommand{\Hom}{\operatorname{Hom}}
\newcommand{\End}{\operatorname{End}}
\newcommand{\Aut}{\operatorname{Aut}}
\newcommand{\Ob}{\operatorname{Ob}}
\newcommand{\Set}{\mathbf{Set}}
\renewcommand{\k}{\Bbbk}

\newcommand{\kk}{R}
\newcommand{\kVect}{\mathbf{Mod_\kk}}
\newcommand{\kVectSum}{\mathbf{Mod_\kk^{\oplus}}}
\newcommand{\kVectGrad}{\mathbf{ModGrad_\kk}}

\newcommand{\Id}{\operatorname{Id}}

\newcommand{\Conj}{\operatorname{Conj}}
\newcommand{\for}{\operatorname{For}}

\newcommand{\one}{\mathbf{1}}

\newcommand{\oV}{\overline{V}}

\renewcommand{\le}{\leqslant}
\renewcommand{\ge}{\geqslant}

\title{Symmetric Categories as a Means of Virtualization for Braid Groups and Racks} 
\author{Victoria LEBED \\ \itshape lebed@math.jussieu.fr}

\begin{document}

\maketitle

\begin{abstract}
\footnotesize
In this paper we propose, firstly, a categorification of virtual braid groups and groupoids in terms of ``locally" braided objects in a symmetric category (SC), and, secondly, a definition of self-distributive structures (SDS) in an arbitrary SC. SDS are shown to produce braided objects in a SC. As for examples, we interpret associativity and the Jacobi identity in a SC as generalized self-distributivity, thus endowing associative and Leibniz algebras with a ``local" braiding. Our ``double braiding" approach provides a natural interpretation for virtual racks and twisted Burau representation. Homology of categorical SDS is defined via weakly simplicial objects, generalizing rack, bar and Leibniz differentials. Free virtual SDS and the faithfulness of corresponding representations of virtual braid groups/groupoids are also discussed.
\end{abstract}



\tableofcontents

\section{Introduction}

The almost century-old \textit{braid theory} is nowadays quite vast and entangled, with unexpected connections with different areas of mathematics still emerging. Dehornoy's spectacular results intertwining self-distributivity, set theory and braid group ordering (\cite{Dehornoy}) provide a good example. \textit{Virtual braid theory} (cf. section \ref{sec:patchwork} for definition and illustrations) dates from the pioneer work of Kauffman (\cite{KauffmanVirtual}) and Vershinin (\cite{Ver}) in the late 1990s, and it still reserves a lot of unexplored questions, in spite of numerous results already obtained. This shows that this theory is far from being an elementary variation of that of usual braids.

Virtual braids are most often considered in the context of virtual knots and links. The topological aspects of these objects are thus naturally in the spotlight. The aim of this paper is on the contrary to clarify some \emph{categorical and representational aspects}.

The flow of ideas related to the objects and concepts studied in this paper can be represented -- very schematically -- by the following chart:

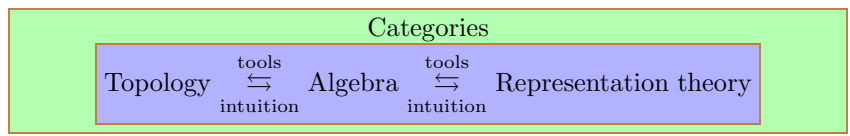
\begin{center}
\fcolorbox{brown}{palegreen}{\begin{minipage}{0.9\textwidth}
\begin{center}
Categories\\
\fcolorbox{brown}{paleblue}{Topology $\underset{\text{intuition}}{\overset{\text{tools}}{\leftrightarrows}}$ Algebra
$\underset{\text{intuition}}{\overset{\text{tools}}{\leftrightarrows}}$ Representation theory}
\end{center}
\end{minipage}}
\end{center}

One has thus a sort of triptych of picturesque mathematical areas, with the category theory as a unifying background. Such unifications are precisely the {raison d'etre} of categories. This chart, certainly subjective and simplified, illustrates quite adequately the content of this paper.

\medskip

We start with an extensive reminder on braid groups $B_n$ in section \ref{sec:braids}, where we extract from the large scope of existing results the ones to be extended to virtual braid groups $VB_n$ in the rest of the paper. Section \ref{sec:patchwork} is a survey of the steps of this vast \textit{virtualization} program which have already been effectuated. Particular attention is given to usual and virtual \textit{SD (= self-distributive) structures} (i.e. shelves, racks and quandles). Virtuality means here the additional datum of a shelf automorphism $f,$ following Manturov (\cite{Manturov}). Subsequent sections contain original patches to the still very fragmentary virtual braid theory.

Section \ref{sec:free} is devoted to \textit{free virtual SD structures}. The faithfulness of the $VB_n^+$ (or $VB_n$) action on these (defined in proposition \ref{thm:VBActsOnVRacks}) is discussed, including a (reformulation of a) conjecture of Manturov (\cite{ManRecognitionEn}). Some arguments in favor of the faithfulness in the free monogenerated virtual shelf case are presented, however without a definite answer.
 
Theorem \ref{thm:AlgCatVirtual} is the heart of the paper. It suggests categorifying virtual braid groups by \textit{``locally" braided objects} (i.e. objects $V$ endowed with an endomorphism $\sigma$ of $V\otimes V$ satisfying the Yang-Baxter equation) in a \textit{``globally" braided symmetric category} $\C.$ Here is the correspondence between the algebraic notion and its categorification:

\begin{center}
\begin{tabular}{| >{\columncolor{paleblue}}c | >{\columncolor{palegreen}}c | >{\columncolor{palegreen}}c |}
\hline
\textbf{category level}     & ``local" braiding   & ``global" symmetric braiding  \\
& for $V$  & on $\C$ \\
\hline
\textbf{$VB_n$ level }  & $B_n$ part  & $S_n$ part\\
\hline
\end{tabular}
\end{center}

Braided objects are a rich source of representations of $VB_n.$
Section \ref{sec:examples} describes examples constructed from shelves and racks, as well as from associative and Leibniz algebras in a symmetric category $(\C, \otimes, \II,c).$ Concretely, ``local" braidings are written in these three cases as follows:
$$\sigma_{\lhd}(a,b)=(b, a\lhd b),$$
$$\sigma_{Ass} = \nu \otimes \mu,$$
$$\sigma_{Lei} = c + \nu \otimes [,],$$
where $\mu$ (resp. $[,]$) is the (Leibniz) multiplication on $V,$ and $\nu: \II \rightarrow V$ is a (Lie) unit. Note that the braiding $\sigma_{Ass}$ is \textit{non-invertible}, providing thus representations of positive virtual braid monoids only (section \ref{sec:positive}). Remark also that our categorification of $VB_n$ is quite different from that of Kauffman and Lambropoulou (\cite{KauffLambr}).

Among the advantages of our ``double braiding approach" is its high flexibility, allowing, for example, to see Manturov's virtual racks via a deformation of the underlying symmetric structure (section \ref{sec:CatForVB}), or to recover the twisted Burau representation (\cite{SW}) by \textit{twisting} the braidings (section \ref{sec:twisted}). 

The second central result of this paper is a definition of \textit{GSD (= generalized self-distributive) structures} in an arbitrary symmetric category (definition \ref{def:shelf_gen}). Our approach is distinct from that of Carter, Crans, Elhamdadi and Saito (\cite{Crans}, \cite{CatSelfDistr}); the difference is discussed in detail in section  \ref{sec:CatSD}. Briefly, we make the comultiplication a part of the GSD structure instead of imposing it on the level of the underlying category: 
\begin{center}
\begin{tabular}{cccccc}
\textbf{GSD} = & comultiplication $\Delta$ &+  & binary operation $\lhd$ &+&compatibility  \\
\hline
&coassociative,&&self-distributive,&&in the braided\\
\multicolumn{2}{r}{weakly cocommutative}&&with $\Delta$ as diagonal&&bialgebra sense\\
\end{tabular}
\end{center}

Shelves and racks in a symmetric category $\C$ are shown to be braided objects in $\C,$ thus giving $VB_n^+$ or $VB_n$ actions. This generalizes the familiar braiding for usual racks. The most naive comultiplications are analyzed, leading to an interpretation of associativity and of the Leibniz condition (= a form of Jacobi identity) as particular cases of generalized self-distributivity. Hopf algebras are also considered in this framework, recovering the braiding known via the Yetter-Drinfeld approach.

In section \ref{sec:hom}, \textit{homologies of GSD structures} are defined in terms of weakly (bi)simplicial structures (cf. theorem \ref{thm:GSDhom}), recovering the familiar rack, bar and Leibniz differentials in suitable settings.

\medskip 
\underline{\textbf{Notations and conventions.}} 

For an object $V$ in a strict monoidal category and a morphism $\varphi:V^{\otimes l}\rightarrow V^{\otimes r},$ the following notations will be repeatedly used:
\begin{equation}\label{eqn:phi_i}
\varphi_i := \Id_V^{\otimes (i-1)}\otimes\varphi \otimes \Id_V^{\otimes (k-i+1)}:V^{\otimes (k+l)}\rightarrow V^{\otimes (k+r)};
\end{equation}
\begin{equation}\label{eqn:phi^i}
\varphi^n := (\varphi_1)^{\circ n}=\varphi_1 \circ \cdots \circ\varphi_1 :V^{\otimes k}\rightarrow V^{\otimes (k+n(r-l))},
\end{equation}
where $\varphi_1$ is composed with itself $n$ times. Further, the abbreviated notation $V^n:= V^{\otimes n}$ will sometimes be used.

The term ``knot" will often be used instead of ``knots and links" for brevity.

The unit interval $[0,1]$ will be denoted by $I.$

Notations $\k$ and $\kk$ will be used for a field and a commutative ring respectively.

\section{Braids: crossings at the intersection of algebra, topology, representation theory and category theory}\label{sec:braids}

The notion of braids, completely intuitive from the topological viewpoint, was first introduced by Emil Artin in 1925, although it was implicitly used by many XIXth century mathematicians. Its algebraic counterpart, the notion of braid groups, accompanies its ``twin brother" from the birth. Representation theory methods have been extensively applied to braids since then, with, as two major examples, Burau representation (and, later, quantum invariants) and Artin action on free groups, both recalled in this section (see for instance  \cite{Birman} for more details). Braided categories, a natural categorification of the braid group, appeared in 1986 (Joyal and Street, \cite{BraidedCat}), long after symmetric categories, corresponding to symmetric groups (Eilenberg and Kelly, \cite{SymmCat}). The aim of this section is to recall all those different viewpoints on braids, before proceeding to their virtualization in the rest of the paper. 

Almost no proofs are given here. For a more detailed exposition, the reader is sent to the wonderfully written books \cite{Birman} and \cite{Braids} for the general aspects of the braid theory, and \cite{Invariants} for the categoric aspects.

 Braid groups may also be interpreted in terms of configuration spaces of $n$ points on the plane, or in terms of mapping class groups of an $n$-punctured disc, but this will not be discussed here.

\subsection{Topology: the birth}  

{Topologically}, a braid can be thought of as a $C^1$ embedding of $n$ copies $I_1, \ldots, I_n$ of $I=[0,1]$ into $\RR^2 \times I,$ with the left ends of the $I_j$'s being sent bijectively to points $(l,0,0), \: 1 \le l \le n$; the right ends being sent bijectively to $(r,0,1), \: 1 \le r \le n$; the tangents being vertical at the endpoints of the $I_j$'s; and the images of the $I_j$'s always looking ``up" (i.e. the embedding $I_j \hookrightarrow \RR^2 \times I$ composed with the projection $\RR^2 \times I\twoheadrightarrow I$ is a homeomorphism). One could also consider smooth or piecewise linear embeddings, with the same resulting theory. Such embeddings, considered up to isotopy, are called \emph{braids on $n$ strands}. They are represented by \emph{diagrams} corresponding to the projection ``forgetting" the second coordinate of $\RR^2 \times I,$ with the under/over information for each crossing:

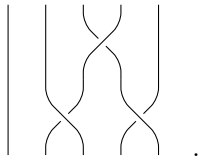
\begin{center}
\begin{tikzpicture}[scale=0.5]
\draw [rounded corners] (1,0) -- (1,4);
\draw [rounded corners] (5,0) -- (5,0.5) -- (4,1.5) --(4,2);
\draw [rounded corners] (4,0) -- (4,0.5) -- (4.4,0.9);
\draw [rounded corners] (5,4) -- (5,1.5) -- (4.6,1.1);
\draw [rounded corners] (3,0) -- (3,0.5) -- (2,1.5) --(2,4);
\draw [rounded corners] (2,0) -- (2,0.5) -- (2.4,0.9);
\draw [rounded corners] (3,2) -- (3,1.5) -- (2.6,1.1);
\draw [rounded corners] (3,2) -- (3,2.5) -- (4,3.5) --(4,4);
\draw [rounded corners] (4,2) -- (4,2.5) -- (3.6,2.9);
\draw [rounded corners] (3,4) -- (3,3.5) -- (3.4,3.1);
\node at (6,0) {.};
\end{tikzpicture}
\end{center}

The vertical ``stacking" of braids on $n$ strands, followed by an obvious contraction, defines a group structure on them. We denote this group by $\B_n.$ 
\begin{center}
\begin{tikzpicture}[scale=0.7]
\draw (0,0) rectangle (1,1);
\node at (0.5,0.5) {$\xi_2$};
\node at (1.5,0.5) {$\cdot$};
\draw (2,0) rectangle (3,1);
\node at (2.5,0.5) {$\xi_1$};
\node at (4,0.5) {$=$};
\draw (5,0) rectangle (6,0.5);
\node at (5.5,0.25) {$\xi_1$};
\draw (5,0.5) rectangle (6,1);
\node at (5.5,0.75) {$\xi_2$};
\end{tikzpicture}
\end{center}

This topological interpretation of braids is particularly useful in knot theory due to the closure operation, effectuated by passing to $\RR^3 \supset \RR^2 \times I$ and connecting all pairs of points $((j,0,0),(j,0,1)),$ where $1 \le j \le n,$ by untangled arcs living ``outside" the braid. 

\begin{center}
\begin{tikzpicture}[scale=0.5]
\draw [rounded corners] (1,2) -- (1,3.8) -- (0,4)-- (-1,3.8)  --(-1,0.2)  -- (0,0) --(1,0.2) --(1,2) ;
\draw [rounded corners] (-5,2) -- (-5,0) -- (0,-0.8) -- (5,0) -- (5,0.5) -- (4,1.5) --(4,2);
\draw [rounded corners] (-4,2) -- (-4,0) -- (0,-0.6)-- (4,0) -- (4,0.5) -- (4.4,0.9);
\draw [rounded corners] (-5,2) -- (-5,4) --(0,4.8) -- (5,4) -- (5,1.5) -- (4.6,1.1);
\draw [rounded corners] (-3,2) -- (-3,0) -- (0,-0.4)-- (3,0) -- (3,0.5) -- (2,1.5) --(2,4)--(0,4.2)--(-2,4)--(-2,2);
\draw [rounded corners] (-2,2) -- (-2,0) -- (0,-0.2)-- (2,0) -- (2,0.5) -- (2.4,0.9);
\draw [rounded corners] (3,2) -- (3,1.5) -- (2.6,1.1);
\draw [rounded corners] (3,2) -- (3,2.5) -- (4,3.5) --(4,4)--(0,4.6)--(-4,4)--(-4,2);;
\draw [rounded corners] (4,2) -- (4,2.5) -- (3.6,2.9);
\draw [rounded corners] (-3,2) -- (-3,4) --(0,4.4) -- (3,4) -- (3,3.5) -- (3.4,3.1);
\end{tikzpicture}
\end{center}

Alexander's theorem (1923) assures that all links and knots are obtained this way, and Markov's theorem (1935) explains which braids give the same link.

\subsection{Algebra: a twin brother}  

{Algebraically}, \emph{the (Artin) braid group} $B_n$ is a kind of generalization of the symmetric group $S_n.$ It is defined by generators $\sigma_1, \sigma_2, \ldots,\sigma_{n-1},$ subject to relations
\begin{align}
\sigma_i \sigma_j & = \sigma_j \sigma_i \qquad\qquad \text{ if } |i-j|>1, 1 \le i,j \le n-1,\label{eqn:BrComm}\tag{$Br_{C}$} \\
\sigma_i \sigma_{i+1} \sigma_i & = \sigma_{i+1} \sigma_i \sigma_{i+1} \qquad \forall \: 1 \le i \le n-2. \label{eqn:BrYB}\tag{$Br_{YB}$}
\end{align}

The first equation means partial commutativity. The second one is a form of \emph{Yang-Baxter equation}.

The symmetric group is then the quotient of $B_n$ by
\begin{equation}\label{eqn:Symm}\tag{$Symm$}
\sigma_i^2 = 1 \qquad \forall \: 1 \le i \le n-1.
\end{equation}
Other quotients of (the group rings of) Artin braid groups are extensively studied by representation theorists. Hecke algebras give a rich example.

The link between algebraic and topological viewpoints was suggested by Emil Artin already in 1925:

\begin{theorem}\label{thm:AlgTop}
There exists an isomorphism $\varphi$ between the groups $B_n$ and $\B_n,$ given by
\begin{center}
\begin{tikzpicture}[scale=0.5]
\node at (-1,1) {$\sigma_i \mapsto $};
\draw (0,0)--(0,2);
\draw (1,0)--(1,2);
\node at (2,0) {$\cdots$};
\draw (3,0)--(3,2);
\draw [rounded corners] (5,0) -- (5,0.5) -- (4.6,0.9);
\draw [rounded corners] (4.4,1.1) -- (4,1.5) --(4,2);
\draw [rounded corners] (4,0) -- (4,0.5) -- (5,1.5) -- (5,2);
\draw (6,0)--(6,2);
\node at (7,0) {$\cdots$};
\draw (8,0)--(8,2);
\node [below] at (0,0) {$\scriptstyle 1$};
\node [below] at (1,0) {$\scriptstyle 2$};
\node [below] at (3,0) {$\scriptstyle {i-1}$};
\node [below,purple] at (4,0) {$\scriptstyle i$};
\node [below,purple] at (5,0) {$\scriptstyle{i+1}$};
\node [below] at (6,0) {$\scriptstyle {i+2}$};
\node [below] at (8,0) {$\scriptstyle n$};
\node at (9,0) {.};
\end{tikzpicture}
\end{center}
\end{theorem}

Thus the Yang-Baxter equation \eqref{eqn:BrYB} is simply the algebraic translation of the third Reidemeister move for braid diagrams:
\begin{center}
\begin{tikzpicture}[scale=0.5]
 \draw [rounded corners] (0,0) -- (2,2) -- (2,3);
 \draw (1,0) -- (0.6,0.4);
 \draw [rounded corners] (0.4,0.6) -- (0,1) -- (0,2) -- (1,3);
 \draw [rounded corners] (2,0) -- (2,1) -- (1.6,1.4);
 \draw [rounded corners] (1.4,1.6) --  (0.6,2.4);
 \draw [rounded corners] (0.4,2.6) --  (0,3);
 \node at (5,1.5) {$=$};
 \node at (8,1) {};
\end{tikzpicture}
\begin{tikzpicture}[scale=0.5]
 \draw [rounded corners] (0,0) -- (0,1) -- (2,3);
 \draw [rounded corners] (1,0) -- (2,1) -- (2,2) -- (1.6,2.4);
 \draw (1.4,2.6) -- (1,3);
 \draw (2,0) -- (1.6,0.4);
 \draw [rounded corners] (1.4,0.6) -- (0.6,1.4);
 \draw [rounded corners] (0.4,1.6) -- (0,2) -- (0,3);
 \node at (4,0) {.};
\end{tikzpicture}
\end{center}

\subsection{Positive braids: little brothers}

In some contexts it is interesting to regard braid diagrams having positive crossings
\begin{tikzpicture}[scale=0.5]
 \draw (0,0) -- (1,1);
 \draw (0,1) -- (0.3,0.7);
 \draw (1,0) -- (0.7,0.3);
\end{tikzpicture}
only. The braids represented by such diagrams are called \emph{positive}. Their interest resides, among other properties, in the fact that every braid is a (non-commutative) quotient of two positive braids. Admitting no inverses, they form a monoid only. The algebraic counterpart is \emph{the positive braid monoid} $B_n^+.$  It is generated - as a monoid this time! - by $\sigma_1, \sigma_2, \ldots,\sigma_{n-1},$ with the same relations as those defining $B_n.$ One can show that this is a submonoid of $B_n.$  A ``positive" analogue of theorem \ref{thm:AlgTop} is obvious.

\begin{remark}\label{rmk:negative}
Considering crossings
\begin{tikzpicture}[scale=0.5]
 \draw (0,1) -- (1,0);
 \draw (0,0) -- (0.3,0.3);
 \draw (1,1) -- (0.7,0.7);
\end{tikzpicture} instead, one gets the notions of negative braids and negative braid monoid $B_n^-,$ isomorphic to $B_n^+$ via the obvious monoid map $\sigma_i^{-1}\mapsto \sigma_i.$
\end{remark}

\subsection{Representations: a full wardrobe}\label{sec:reps}

Algebraic structures, even those admitting easy descriptions, are often difficult to study using algebraic tools only. Even comparing two elements of a group defined by generations and relations can be a hard task. A recurrent solution is exploring representations (linear or more general) of algebraic objects instead, i.e., in a metaphorical language, looking for fitting clothes. Free actions are of particular interest, since they allow one to easily distinguish different elements. The corresponding concept on the topological level is that of invariants.

Note that symmetric groups are even \underline{defined} via their action on a set. This suggests the importance of braid group representations, two of which are recalled here.

\medskip

The first one was discovered by Burau as early as in 1936.
\begin{proposition}
An action of the braid group $B_n$ on $\ZZ[t^{\pm 1}]^{\otimes n}$ can be given, in the matrix form, by 
\begin{align}
\rho(\sigma_i)&=\begin{pmatrix} I_{i-1} & 0&0&0 \\ 0 &0&1&0\\0&t&1-t&0\\0&0&0&I_{n-i-1}\end{pmatrix},\label{eqn:Burau}\\
\rho(\sigma_i^{-1})&=\begin{pmatrix} I_{i-1} & 0&0&0 \\0&1-t^{-1}&t^{-1}&0\\ 0 &1&0&0\\0&0&0&I_{n-i-1}\end{pmatrix}.\notag
\end{align}
\end{proposition}

This representation is interesting from the topological viewpoint: the Alexander polynomial, a famous knot invariant, can be interpreted as $\det(\II - \rho_*(b)),$ where $b$ is any braid whose closure gives the knot in question, and $\rho_*$ is (a reduced version of) Burau representation.

Burau representation, conjectured to be faithful for a long time, turns out not to be so for $n \ge 5.$ A faithful linear representation was found later by Lawrence, Krammer and Bigelow.

\medskip

Another representation - a faithful one this time - was already known by Artin. See \cite{Birman} or \cite{FR} for a topological proof, or \cite{Dehornoy} for a short algebraic one.

\begin{theorem}\label{thm:ActionOnFGn}
Denote by $F_n$ the free group with $n$ generators $x_1, \ldots, x_n.$ The braid group $B_n$ faithfully acts on $F_n$ according to the formulas
$$\sigma_i (x_j) = \begin{cases}
 x_{i} & \text{if } j=i+1, \\
 x_{i} x_{i+1} x_{i}^{-1} & \text{if } j=i, \\
 x_j & \text{otherwise;} 
\end{cases} \qquad \sigma_i^{-1} (x_j) = \begin{cases}
 x_{i+1} & \text{if } j=i, \\
 x_{i+1}^{-1} x_i x_{i+1} & \text{if } j=i+1, \\
 x_j & \text{otherwise.} 
\end{cases}$$
\end{theorem}

These two seemingly different representations will be interpreted as particular cases of a much more general one described in the next section.

\subsection{Shelves and racks: arranging the wardrobe}\label{sec:racks}

Here we give an example of a topological idea inspiring important algebraic structures -- that of shelves, racks and quandles -- with an extremely rich representation theory. 

For a detailed introduction to the theory of racks and quandles, as well as for numerous examples, we send the reader to seminal papers \cite{Joyce} and \cite{Matveev}, or to \cite{Kamada} and \cite{Crans} for very readable surveys. The term ``shelf" is in particular borrowed from Crans's thesis \cite{Crans}. The connections between racks and braids are explored in detail in \cite{FR}. As for free self-distributive structures applied to braids, \cite{Dehornoy} and \cite{KasselDehornoy} are nice sources.

\begin{definition}
A \emph{shelf} is a set $S$ with a binary operation $ \lhd : S \times S \longrightarrow S$ satisfying the self-distributivity condition
\begin{equation} \label{eqn:SelfDistr}\tag{SD}
  (a\lhd b)\lhd c = (a\lhd c)\lhd(b\lhd c) \qquad \forall a,b,c \in S.
\end{equation} 
If moreover the application $a\mapsto a\lhd b$ is a bijection on $S$ for every $b \in S,$ that is if there exists an ``inverse" application $\widetilde{\lhd}: S \times S\longrightarrow S$ such that
\begin{equation}\label{eqn:Rack}\tag{R}
  (a\lhd b)\widetilde{\lhd} b = (a\widetilde{\lhd} b)\lhd b = a \qquad \forall a,b \in S,
\end{equation} 
then the couple $(S,\lhd)$ is called \emph{a rack}.
A \emph{quandle} is a rack satisfying moreover
\begin{equation}\label{eqn:Quandle}\tag{Q}
  a\lhd a = a \qquad \forall a \in S.
\end{equation} 
\end{definition}

The term \emph{SD structures} will be used here to refer to any of these three structures, underlining the importance of relation \eqref{eqn:SelfDistr}.

There are numerous examples of SD structures coming from various areas of mathematics. Only several of them are relevant here, allowing us to recover, after a short categorical reminder in the next section, 
the two actions from section \ref{sec:reps}. Free SD structures are also presented here for further use.

\begin{enumerate}
\item The one-element shelf is necessarily a quandle, called the \emph{trivial quandle}.
\item A group $G$ can be endowed with, among others, a \emph{conjugation quandle} structure:
\begin{align*}
a \lhd b & := b^{-1}ab,\\
a \wlhd b &:= bab^{-1}.
\end{align*}
This quandle is denoted by $\Conj(G).$
Morally, the quandle structure captures the properties of conjugation in a group, forgetting the multiplication structure it comes from. This is the motivating example on the algebra side.
\item The \emph{Alexander quandle} is the set $\ZZ[t^{\pm 1}]$ with the operations
\begin{align}
a \lhd b &= ta + (1-t)b,\label{eqn:Alex}\\
a \wlhd b &= t^{-1}a + (1-t^{-1})b.\notag
\end{align}
\item The \emph{cyclic rack} $CR$ is the set of integer numbers $\ZZ$ with the operations
\begin{align*}
n \lhd m &= n+1 \qquad \forall n,m \in \ZZ,\\
n \wlhd m &= n-1 \qquad \forall n,m \in \ZZ.
\end{align*}
Note that it is very far from being a quandle: the property \eqref{eqn:Quandle} is false for all the elements. Moreover, the quotient of $\ZZ$ by \eqref{eqn:Quandle} is the trivial quandle. This somewhat exotic structure will be interpreted in the next example.
\item Given a set $X,$ the \emph{free shelf, rack or quandle} on $X$ are denoted by $FS(X),FR(X),FQ(X)$ respectively. This is abbreviated to
$$FS_n:=FS(\{x_1,\ldots,x_n\}),$$
$$FS_{\ZZ}:=FS(\{x_i, i\in \ZZ \}),$$
and similar for racks and quandles.

 Free shelves, even generated by one element only, are extremely complicated structures. They have in particular allowed Dehornoy to construct in the early 90's a total left-invariant group order on $B_n$ (see \cite{KasselDehornoy} or \cite{Dehornoy} for a nice exposition). Monogenerated quandles are trivial, since $x \lhd x = x$ for the generator $x.$ As for racks, which are intermediate objects between shelves and quandles, the structure is quite simple but not trivial:
\begin{align}\label{ref:FR1}
FR_1 &\stackrel{\sim}{\longrightarrow} CR,\\
((x \lhd \ x) \lhd \cdots ) \lhd x &\longmapsto n,\notag\\
((x \wlhd \ x) \wlhd \cdots) \wlhd x &\longmapsto -n,\notag
\end{align}
where $n$ is the number of operations $\lhd$ (resp. $\wlhd$) in the expression.

Note that $FQ(X)$ becomes interesting for larger sets $X.$ In particular, $FQ_n$ can be described via the quandle injection
\begin{align}
FQ_n & {\hookrightarrow} \Conj(F_n),\label{eqn:FQ}\\
x_i &\longmapsto x_i.\notag
\end{align}
The image of this injection is the sub-quandle of $\Conj(F_n)$ generated by the $x_i$'s. To see this, remark that
\begin{itemize}
\item  any element of $FQ_n$ can be written in the form $$((x_{i_0} \lhd^{\varepsilon_1} x_{i_1})\lhd^{\varepsilon_2} \cdots )\lhd^{\varepsilon_k} x_{i_k},$$ where the values of the $\varepsilon_j$'s are $\pm 1$'s, with notations $$\lhd^{1} = \lhd,\qquad \lhd^{-1} = \wlhd;$$
\item in $\Conj(G),$ one has $a \wlhd b = a \lhd b^{-1}.$
\end{itemize}
\end{enumerate}

The theory of racks and quandles owes its rising popularity to topological applications: it allows to upgrade the fundamental group of the complement of a knot to a complete invariant (up to a symmetry). The connection to groups is clear from the example of the conjugation quandle. The connection to knots and braids is illustrated by the following well-known result, which will be better explained later.

\begin{proposition}\label{thm:SDaction}
Take a shelf $(S,\lhd).$ An action of the positive braid monoid $B_n^+$ on $S^{\times n}$ can be given as follows:
$$\sigma_i(a_1,\ldots,a_n)= (a_1,\ldots,a_{i-1},a_{i+1},a_{i} \lhd a_{i+1},a_{i+2},\ldots,a_n).$$
If $S$ is a rack, then this action becomes a braid group action:
$$\sigma_i^{-1}(a_1,\ldots,a_n)= (a_1,\ldots,a_{i-1},a_{i+1}\widetilde{\lhd} a_{i},a_{i},a_{i+2},\ldots,a_n).$$
More precisely, these formulas define an action of $B_n^+$ (resp. $B_n$) if and only if the couple $(S,\lhd)$ satisfies the shelf (resp. rack) axioms.
\end{proposition}

The racks are thus the ``right" structure for carrying a braid group action.

This action is diagrammatically depicted via braid coloring:

\begin{center}
\begin{tikzpicture}[scale=0.5]
 \draw [->, thick]  (-4,0.2) -- (-4,0.8);
 \draw (0,0) -- (1,1);
 \draw (1,0) -- (0.6,0.4);
 \draw (0.4,0.6) -- (0,1);
 \node at (1,0) [below] {$b$};
 \node at (0,0) [below] {$a$};
 \node at (0,1) [above] {$b$}; 
 \node at (1,1) [above right] {$a \lhd b$};
 \node at (4,1) {};
\end{tikzpicture}
\begin{tikzpicture}[scale=0.5]
 \draw (1,0) -- (0,1);
 \draw (0,0) -- (0.4,0.4);
 \draw (0.6,0.6) -- (1,1);
 \node at (1,0) [below] {$b$};
 \node at (0,0) [below] {$a$};
 \node at (0,1) [above] {$b \wlhd a$}; 
 \node at (1,1) [above right] {$a$};
 \node at (2,0)  {.};
\end{tikzpicture}
\end{center}

Thus when a strand passes under another one, the corresponding element of $S$ acts by $\lhd$ or $\wlhd$ on the element on the higher strand. In some references it happens the other way around, which is simply a matter of choice: one could have interchanged the actions of the $\sigma_i$'s and their inverses.

Note that, with our choice of composing braids from bottom to top, the actions will always be depicted from bottom to top here, which is indicated by the arrow on the above diagram.

Let us finish with some faithfulness remarks for free SD structures.

\begin{proposition}\label{thm:FSActsFreely}
The action of the positive braid monoid $B_n^+$ on $FS_1^{\times n}$ is free.
\end{proposition}

\begin{proof}
This is an easy consequence of Dehornoy's results on the ordering of free shelves and braid groups (cf. \cite{KasselDehornoy} and \cite{Dehornoy}). He shows that $FS_1$ has a total order $<$ generated by the partial order $\prec$:
 $$a = c \lhd b \;\Longrightarrow \; b \prec a.$$ This induces an anti-lexicographic order $<$ on $FS_1^{\times n}.$ (The prefix ``anti" is needed here since we consider the right version of self-distributivity, while Dehornoy works with the left one.) Then he proves that this order is sufficiently nice to induce a total order on the group $B_n$ by declaring, for $\alpha, \beta \in B_n,$
$$\alpha < \beta \; \Longleftrightarrow\; \forall \overline{a} \in FS_1^{\times n}, \; \alpha(\overline{a}) < \beta (\overline{a}).$$
Here $\alpha(\overline{a})$ and $\beta (\overline{a})$ denote a partial extension to $B_n$ of the action of $B_n^+,$ and one takes only those $\overline{a}$ for which $\alpha(\overline{a})$ and $\beta (\overline{a})$ are both defined (they exist).
 Thus, if $\alpha$ and $\beta$ act on an element of $FS_1^{\times n}$ in the same way, this means precisely $\alpha = \beta.$
\end{proof}

\begin{proposition}\label{thm:FQActsFaith}
The action of the braid group $B_n$ on $FQ_n^{\times n}$ is faithful.
\end{proposition}

\begin{proof}
One remarks, for a braid $\alpha \in B_n,$ the identity
$$\alpha(x_1, \ldots, x_n) = (\alpha^{-1}(x_1), \ldots, \alpha^{-1}(x_n)),$$
where the action on the right is that of theorem \ref{thm:ActionOnFGn}, which is faithful, thus permitting to conclude.
\end{proof}

Remark that this action is not free, since, for instance, the elements $1$ and $\sigma_1$ of $B_n$ act in the same way on diagonal elements $(a,a,\ldots,a) \in FQ_n^{\times n}, \; a \in FQ_n.$

Since $FQ(X)$ is the quotient of $FR(X)$ by \eqref{eqn:Quandle}, one has
\begin{corollary}
The action of $B_n$ on $FR_n^{\times n}$ is faithful.
\end{corollary}

Observe that a monogenerated free rack is not sufficient to produce a faithful action: passing to $CR,$ via the isomorphism \eqref{ref:FR1}, one sees that the action of $B_n$ on $CR^{\times n}$ simply counts the algebraic number of times a strand passes over other strands (with the sign $+$ when moving to the right and $-$ when moving to the left); thus the following two braids are indistinguishable by their action:

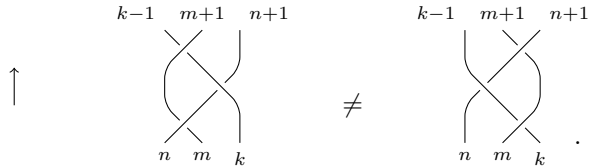
\begin{center}
\begin{tikzpicture}[scale=0.5]
 \draw [->]  (-4,1) -- (-4,2);
 \draw [rounded corners] (0,0) -- (1.4,1.4);
 \draw [rounded corners] (2,3) -- (2,2) -- (1.6,1.6);
 \draw (1,0) -- (0.6,0.4);
 \draw [rounded corners] (0.4,0.6) -- (0,1) -- (0,2) -- (1,3);
 \draw [rounded corners] (2,0) -- (2,1) -- (0.6,2.4);
 \draw [rounded corners] (0.4,2.6) --  (0,3);
 \node at (2,0) [below] {$\scriptstyle{k}$};
 \node at (1,0) [below] {$\scriptstyle{m}$};
 \node at (0,0) [below] {$\scriptstyle{n}$};
 \node at (2,3) [above right] {$\scriptstyle{n+1}$};
 \node at (1,3) [above] {$\scriptstyle{m+1}$};
 \node at (0,3) [above left] {$\scriptstyle{k-1}$};
 \node at (5,1) {$\neq$};
 \node at (6,1) {};
\end{tikzpicture}
\begin{tikzpicture}[scale=0.5]
 \draw [rounded corners] (0,0) -- (0,1) -- (0.4,1.4);
 \draw (0.6,1.6) -- (2,3);
 \draw [rounded corners] (1,0) -- (2,1) -- (2,2) -- (1.6,2.4);
 \draw (1.4,2.6) -- (1,3);
 \draw (2,0) -- (1.6,0.4);
 \draw [rounded corners] (1.4,0.6) -- (0,2) -- (0,3);
 \node at (2,0) [below] {$\scriptstyle{k}$};
 \node at (1,0) [below] {$\scriptstyle{m}$};
 \node at (0,0) [below] {$\scriptstyle{n}$};
 \node at (2,3) [above right] {$\scriptstyle{n+1}$};
 \node at (1,3) [above] {$\scriptstyle{m+1}$};
 \node at (0,3) [above left] {$\scriptstyle{k-1}$};
 \node at (3,0) {.};
\end{tikzpicture}
\end{center}

\subsection{Categories: maturity}\label{sec:cat} 

The notion of braids is very ``categorical" - more than that of knots for example. Braids naturally ``correspond" (in the sense specified later) to the notion of a \emph{braided monoidal category}. The latter is defined as a strict monoidal category $\C$ endowed with a \emph{braiding}, i.e. a natural family of isomorphisms
 $$c=\{c_{V,W} : V\otimes W \stackrel{\sim}{\longrightarrow} W \otimes V\}, \quad \forall V,W \in Ob(\C),$$
 satisfying
 \begin{align}
c_{V,W\otimes U} &=(\Id_W \otimes c_{V,U} )\circ(c_{V,W}\otimes \Id_U),\label{eqn:br_cat}\\
c_{V\otimes W, U} &=(c_{V,U}\otimes \Id_W)\circ(\Id_V \otimes c_{W,U})\label{eqn:br_cat2}
 \end{align}
for any triple of objects $V, W, U.$ ``Natural" means here
\begin{equation}\label{eqn:nat}\tag{Nat}
c_{V',W'} \circ(f\otimes g) = (g\otimes f)\circ c_{V,W}
\end{equation}
 for all $ V,W,V',W' \in Ob(\C), f \in \Hom_\C(V,V'), g \in \Hom_\C(W,W').$
 
 A braided category $\C$ is called \emph{symmetric} if its braiding is symmetric:
\begin{equation}\label{eqn:symm_cat}
c_{V,W} \circ c_{W,V} =\Id_{W\otimes V}, \quad \forall V,W \in Ob(\C).
\end{equation}
 
 The part ``monoidal" of the usual terms ``braided monoidal" and ``symmetric monoidal" will be omitted in what follows.
 
 We work only with \underline{strict} monoidal categories here for the sake of simplicity; according to a theorem of Mac Lane (\cite{Cat}), each monoidal category is monoidally equivalent to a strict one.

Now, two important classical results express a deep connection between the notions of braids and braided categories.
 
\begin{theorem}\label{thm:AlgCat}
Denote by $\Cbr$ the free braided category generated by a single object $V.$ Then for each $n$ one has a group isomorphism
\begin{align}
\psi: B_n &\stackrel{\sim}{\longrightarrow} \End_{\Cbr}(V^{\otimes n}),\notag \\
\sigma_i^{\pm 1} &\longmapsto \Id_V^{\otimes (i-1)}\otimes c_{V,V}^{\pm 1} \otimes \Id_V^{\otimes (n-i-1)}.\label{eqn:br_grp_cat} 
\end{align}
\end{theorem}

Thus braid groups describe hom-sets of a free monogenerated braided category.

The second result is the following:

\begin{corollary}
For any object $V$ in a braided category $\C,$ the map defined by formula \eqref{eqn:br_grp_cat} endows $V^{\otimes n}$ with an action of the group $B_n.$ 
\end{corollary}

This corollary is a major source of representations of the braid group.

Theorems \ref{thm:AlgTop} and \ref{thm:AlgCat} put together give

\begin{corollary}\label{thm:CatTop}
The category $\Cbr$ is isomorphic, as a braided category, to the category $\mathcal B r$ of braids (objects = $\NN$, $\End_{\mathcal B r}(n)$ = braids on $n$ strands, $c_{1,1} =$\begin{tikzpicture}[scale=0.5]
 \draw (0,0) -- (1,1);
 \draw (0,1) -- (0.3,0.7);
 \draw (1,0) -- (0.7,0.3);
\end{tikzpicture}).
\end{corollary}


\medskip

Let us now list several basic examples of symmetric categories which will be useful for the rest of the paper. Notation $\kk$ will always stand for a commutative ring.

\begin{itemize}
\item The category of sets $\Set$ is monoidal, with the Cartesian product $\times$ as its tensor product, and a one-element set $\II$ as its identity object. The identification of $(A\times B)\times C$ with $A\times (B\times C)$ and of $\II\times A$ with $A\times \II$ and with $A$ for any sets $A,B,C,$ which will be implicit in what follows, gives a strict monoidal category. This category is symmetric, with the braiding given by the usual \emph{flip} isomorphism
$$\tau : (a, b)\longmapsto (b, a).$$
\item The category $\kVect$ of $\kk$-modules and $\kk$-linear maps can be seen as symmetric in two ways. Firstly, it can be endowed with the usual tensor product, the one-dimensional free module $\kk$ as its identity object and the flip 
$$\tau : v\otimes w\longmapsto w\otimes v$$
as its braiding. Secondly, one can take the direct sum $\oplus$ as a tensor product, the zero module as its identity object and the flip 
$$\tau : v\oplus w\longmapsto w\oplus v$$
as its braiding. Notation $\kVectSum$ will be used for the second structure. Identifications similar to those for sets are implicit in both cases to assure the strictness. 

The linearization map gives a functor of symmetric monoidal categories 
\begin{align}\label{eqn:Lin}
Lin:\Set &\longrightarrow \kVect,\\
S &\longmapsto \kk S.\notag
\end{align}
One more functor of symmetric monoidal categories, the forgetful one, will be used:
\begin{align}\label{eqn:For}
For :\kVectSum &\longrightarrow \Set,\\
V &\longmapsto V.\notag
\end{align}

Note that both functors are faithful but not full in general.

\item The category of $\ZZ$-graded $\kk$-modules $\kVectGrad$ is symmetric, with the usual graded tensor product, the zero-graded module $\kk$ as its identity object and the \textit{Koszul flip}
$$\tau_{Koszul} : v\otimes w\longmapsto (-1)^{\deg v \deg w} w\otimes v$$
for homogeneous $v$ and $w$ as its braiding. Necessary identifications are effectuated to assure the strictness. Note that this last braiding explains the {Koszul sign convention} in many settings.
\item
One can change the sign $(-1)^{\deg v \deg w}$ in the definition of Koszul flip by any other antisymmetric bicharacter.  
Take in particular a finite abelian group $\Gamma$ endowed with an antisymmetric bicharacter $\chi.$ The category $_\Gamma\!\kVect$ of $\kk$-modules graded over $\Gamma$ is symmetric, with the usual $\Gamma$-graded tensor product, the zero-graded $\kk$ as its identity object and, as a braiding, the so-called \textit{color flip}
$$\tau_{color} : v\otimes w\longmapsto \chi(f,g) w\otimes v$$
for homogeneous $v$ and $w$ graded over $f$ and $g \in \Gamma$ respectively.
\end{itemize}

In what follows the categories from this list will be implicitly endowed with the symmetric structures described here.

\section{Moving to virtual reality}

\subsection{A patchwork of existing concepts and results}\label{sec:patchwork}

The concept of virtuality was born in the \underline{topological} framework in Kauffman's pioneer 1999 paper \cite{KauffmanVirtual} (announced in 1996). See also \cite{Nelson} for an express introduction. The original idea is very natural. One tries to code a knot by writing down the sequence of its crossings encountered when moving along a diagram of the knot, with additional under/over and orientation information for each crossing. This code, called \emph{Gauss code}, is unambiguous but not surjective: some sequences do not correspond to any knot, since while decoding them one may be forced to intersect the part of the diagram drawn before. Kauffman's idea was to introduce in this situation a new, \emph{virtual} type of crossings in a diagram. They are depicted like this: \begin{tikzpicture}[scale=0.5]
 \draw (0,0) -- (1,1);
 \draw (0,1) -- (1,0);
 \draw (0.5,0.5) circle (0.2);
\end{tikzpicture}. Such crossings ``are not here", they come from the necessity to draw in the plane a diagram given abstractly by its Gauss code. The same happens when one has to draw an abstract non-planar graph in $\RR^2.$ Note that the ``under / over" distinction is no longer relevant for virtual crossings.

Another situation where \emph{virtual knots}, i.e. knots with both usual and virtual crossings, naturally emerge is  when one wants to depict in $\RR^2$ knot diagrams living on surfaces other than the plane (for instance, on a torus). See also an interpretation in terms of abstract knot diagrams (\cite{Kamada2}).

A virtual theory parallel to that of classical knots has been developed in numerous papers. We extract from it only the part concerning \emph{virtual braids}, essentially due to Vershinin (see his 1998 paper \cite{Ver}). To emphasize the connexion to virtual knots, one notes the Alexander-Markov type result of Kamada (\cite{KamadaVirtual}) describing the closure operation for virtual braids. 
 
Unfortunately there seem to be no purely topological elementary definition of virtual braids. The common definition is combinatorial: one considers braid diagrams with usual and virtual crossings up to certain relations, which are versions of Reidemeister moves and which are dictated by the Gauss coding. Here is an example of a  relation involving both usual and virtual crossings -- the mixed Yang-Baxter relation:
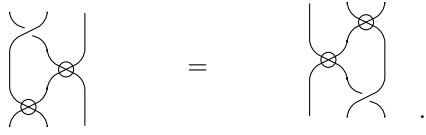
\begin{center}
\bigskip
\begin{tikzpicture}[scale=0.5]
\draw [rounded corners](0,0)--(0,0.25)-- (1,0.75)--(1,1.25)--(2,1.75)--(2,3);
\draw [rounded corners](1,0)--(1,0.25)--(0,0.75)--(0,2.25)--(1,2.75)--(1,3);
\draw [rounded corners](2,0)--(2,1.25)--(1,1.75)--(1,2.25)--(0.6,2.4);
\draw [rounded corners](0.4,2.6)--(0,2.75)--(0,3); 
\draw (0.5,0.5) circle (0.2);
\draw (1.5,1.5) circle (0.2);
\node  at (5,1.5){$=$};
\end{tikzpicture}
\begin{tikzpicture}[scale=0.5]
\node  at (-2,1.5){};
\draw [rounded corners](1,1)--(1,1.25)--(2,1.75)--(2,3.25)--(1,3.75)--(1,4);
\draw [rounded corners](0,1)--(0,2.25)--(1,2.75)--(1,3.25)--(2,3.75)--(2,4);
\draw [rounded corners](2,1)--(2,1.25)--(1.6,1.4);
\draw [rounded corners](1.4,1.6)--(1,1.75)--(1,2.25)--(0,2.75)--(0,4);
\draw (1.5,3.5) circle (0.2);
\draw (0.5,2.5) circle (0.2);
\node  at (3,1){.};
\end{tikzpicture}
\end{center}

It is thus natural to start from the \underline{algebraic} viewpoint.

\begin{definition}\label{def:VB}
The \emph{virtual braid group} $VB_n$ is defined by a set of generators $\{\sigma_i,\zeta_i, \;  1 \le i \le n-1\},$ and the following relations:
\begin{enumerate}
\item \eqref{eqn:BrComm} and \eqref{eqn:BrYB} for the $\sigma_i$'s;
\item \eqref{eqn:BrComm}, \eqref{eqn:BrYB} and \eqref{eqn:Symm} for the $\zeta_i$'s;
\item  \emph{mixed relations}
\begin{align}
\sigma_i \zeta_j & = \zeta_j \sigma_i \qquad\qquad \text{  if } |i-j|>1, 1 \le i,j \le n-1,\label{eqn:BrCommMixed}\tag{$Br_{C}^m$} \\
\zeta_i \zeta_{i+1} \sigma_i & = \sigma_{i+1} \zeta_i \zeta_{i+1} \qquad \forall \: 1 \le i \le n-2. \label{eqn:BrYBMixed}\tag{$Br_{YB}^m$}
\end{align}
\end{enumerate}
\end{definition}

In other words, the group $VB_n$ is the direct product $B_n\ast S_n$ factorized by relations \eqref{eqn:BrCommMixed} and \eqref{eqn:BrYBMixed}. This explains the name ``braid-permutation group" used in \cite{FRR} for a slightly different, but closely related structure.

\medskip

Now \emph{virtual braids on $n$ strands} can be (rather informally) defined as the monoid of braid diagrams with usual and virtual crossings up to ambient isotopy, factorized by the kernel of the monoid surjection
\begin{center}
\begin{tikzpicture}[scale=0.5]
\draw (1,0)--(1,2);
\node at (2,0) {$\cdots$};
\draw (3,0)--(3,2);
\draw [rounded corners] (5,0) -- (5,0.5) -- (4.6,0.9);
\draw [rounded corners] (4.4,1.1) -- (4,1.5) --(4,2);
\draw [rounded corners] (4,0) -- (4,0.5) -- (5,1.5) -- (5,2);
\draw (6,0)--(6,2);
\node at (7,0) {$\cdots$};
\draw (8,0)--(8,2);
\node [below] at (1,0) {$\scriptstyle 1$};
\node [below] at (3,0) {$\scriptstyle {i-1}$};
\node [below,purple] at (4,0) {$\scriptstyle i$};
\node [below,purple] at (5,0) {$\scriptstyle{i+1}$};
\node [below] at (6,0) {$\scriptstyle {i+2}$};
\node [below] at (8,0) {$\scriptstyle n$};
\node at (9,1) {$\mapsto \sigma_i,$};
\node at (12,1) {};
\end{tikzpicture} \begin{tikzpicture}[scale=0.5]
\draw (1,0)--(1,2);
\node at (2,0) {$\cdots$};
\draw (3,0)--(3,2);
\draw [rounded corners] (5,0) -- (5,0.5) -- (4,1.5) --(4,2);
\draw [rounded corners] (4,0) -- (4,0.5) -- (4.4,0.9);
\draw [rounded corners] (5,2) -- (5,1.5) -- (4.6,1.1);
\draw (6,0)--(6,2);
\node at (7,0) {$\cdots$};
\draw (8,0)--(8,2);
\node [below] at (1,0) {$\scriptstyle 1$};
\node [below] at (3,0) {$\scriptstyle {i-1}$};
\node [below,purple] at (4,0) {$\scriptstyle i$};
\node [below,purple] at (5,0) {$\scriptstyle{i+1}$};
\node [below] at (6,0) {$\scriptstyle {i+2}$};
\node [below] at (8,0) {$\scriptstyle n$};
\node at (10,1) {$\mapsto \sigma_i^{-1},$};
\end{tikzpicture}

\begin{tikzpicture}[scale=0.5]
\draw (1,0)--(1,2);
\node at (2,0) {$\cdots$};
\draw (3,0)--(3,2);
\draw [rounded corners] (5,0) -- (5,0.5) -- (4,1.5) --(4,2);
\draw [rounded corners] (4,0) -- (4,0.5) -- (5,1.5) -- (5,2);
\draw (4.5,1) circle (0.3);
\draw (6,0)--(6,2);
\node at (7,0) {$\cdots$};
\draw (8,0)--(8,2);
\node [below] at (1,0) {$\scriptstyle 1$};
\node [below] at (3,0) {$\scriptstyle {i-1}$};
\node [below,purple] at (4,0) {$\scriptstyle i$};
\node [below,purple] at (5,0) {$\scriptstyle{i+1}$};
\node [below] at (6,0) {$\scriptstyle {i+2}$};
\node [below] at (8,0) {$\scriptstyle n$};
\node at (9,1) {$\mapsto \zeta_i.$};
\end{tikzpicture}
\end{center}
The (evident) definition of the monoid of braid diagrams with usual and virtual crossings is omitted here for the sake of concision.

Observe that virtual braids inherit a group structure from $VB_n.$ Note also that theorem \ref{thm:AlgTop} becomes a definition in the virtual world. In what follows virtual braids will be identified with corresponding elements of $VB_n.$

\begin{remark}\label{rmk:OtherYB}
In $VB_n$ one automatically has two other versions of Yang-Baxter relation with one $\sigma$ and two $\zeta$'s. On the contrary, Yang-Baxter relations with one $\zeta$ and two $\sigma$'s do not hold. It comes from the fact Gauss decoding process unambiguously prescribes the pattern of usual crossings and leaves a certain liberty only in placing virtual crossings (recall that the definition of virtual knots was motivated by Gauss coding). Such YB relations are called \emph{forbidden}. Here is an example:
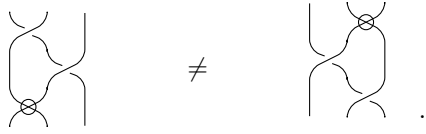
\begin{center}
\bigskip
\begin{tikzpicture}[scale=0.5]
\draw [rounded corners](0,0)--(0,0.25)-- (1,0.75)--(1,1.25)--(2,1.75)--(2,3);
\draw [rounded corners](1,0)--(1,0.25)--(0,0.75)--(0,2.25)--(1,2.75)--(1,3);
\draw [rounded corners](2,0)--(2,1.25)--(1.6,1.4);
\draw [rounded corners](1.4,1.6)--(1,1.75)--(1,2.25)--(0.6,2.4);
\draw [rounded corners](0.4,2.6)--(0,2.75)--(0,3); 
\draw (0.5,0.5) circle (0.2);
\node  at (5,1.5){$\neq$};
\end{tikzpicture}
\begin{tikzpicture}[scale=0.5]
\node  at (-2,1.5){};
\draw [rounded corners](1,1)--(1,1.25)--(2,1.75)--(2,3.25)--(1,3.75)--(1,4);
\draw [rounded corners](0,1)--(0,2.25)--(1,2.75)--(1,3.25)--(2,3.75)--(2,4);
\draw [rounded corners](2,1)--(2,1.25)--(1.6,1.4);
\draw [rounded corners](1.4,1.6)--(1,1.75)--(1,2.25)--(0.6,2.4);
\draw [rounded corners] (0.4,2.6)--(0,2.75)--(0,4);
\draw (1.5,3.5) circle (0.2);
\node  at (3,1){.};
\end{tikzpicture}
\end{center}
\end{remark}

\medskip
As for \underline{representations}, shelves and racks remain relevant in the virtual world:
\begin{proposition}\label{thm:VBActsOnRacks}
 Given a rack $(S,\lhd),$ the virtual braid group $VB_n$ acts on $S^{\times n}$ as follows:
\begin{align}
\zeta_i(a_1,\ldots,a_n)&= (a_1,\ldots,a_{i-1},a_{i+1},a_{i},a_{i+2},\ldots,a_n),\label{eqn:VBactionOnRacks}\\
\sigma_i(a_1,\ldots,a_n)&= (a_1,\ldots,a_{i-1},a_{i+1},a_{i} \lhd a_{i+1},a_{i+2},\ldots,a_n).\label{eqn:VBactionOnRacks2}
\end{align}
\end{proposition}

Note that one can not hope such actions to be faithful, since virtual braids $\sigma_i \sigma_{i+1} \zeta_i$ and $\zeta_{i+1} \sigma_i \sigma_{i+1},$ implied in the forbidden YB relation depicted above, act on $S^{\times n}$ in the same way. 

Manturov proposed in 2002 (cf. \cite{Manturov}) a more adequate structure, called a virtual quandle. We recall it here, as well as its non-idempotent and  non-invertible analogues.

\begin{definition}\label{def:VSD}
A \emph{virtual shelf} is a shelf $(S,\lhd)$ endowed with a shelf automorphism $f:S\longrightarrow S,$ i.e. 
\begin{enumerate}
\item  $f$ admits an inverse $f^{-1},$
\item $f(a\lhd b)=f(a)\lhd f(b).$
\end{enumerate}
If moreover $(S,\lhd)$ is a rack or a quandle, then the triple $(S,\lhd,f)$ is called a \emph{virtual rack / quandle}.
\end{definition}

Note that for a virtual rack, one automatically has
$$f(a \widetilde{\lhd} b)=f(a) \widetilde{\lhd} f(b).$$

The actions from proposition \ref{thm:VBActsOnRacks} can now be upgraded as follows: 
\begin{proposition}\label{thm:VBActsOnVRacks}
 Given a virtual rack $(S,\lhd,f),$ the virtual braid group $VB_n$ acts on $S^{\times n}$ by
\begin{equation}\label{eqn:VBactionOnVRacks}
\zeta_i(a_1,\ldots,a_n)= (a_1,\ldots,a_{i-1},f^{-1}(a_{i+1}),f(a_{i}),a_{i+2},\ldots,a_n)
\end{equation}
and \eqref{eqn:VBactionOnRacks2}.
\end{proposition}

We finish this section with two examples proposed by Manturov (cf. \cite{Manturov}).
\begin{example}\label{ex:ConjVR}
In an arbitrary rack $R,$ right adjoint action by an element $a,$ given by $b\mapsto b \lhd a$ for all $b \in R,$ is a rack automorphism. Its inverse is $b\mapsto b \wlhd a.$ Thus $\Conj(F_n)$ can be given a virtual quandle structure by
$$f(a):=a \lhd x_{n} \qquad \forall a \in F_n.$$
This virtual quandle will appear in the context of free SD structures.
\end{example}
\begin{example}\label{ex:AlexVirt}
The Alexander quandle \eqref{eqn:Alex} can be endowed with two virtual quandle structures.
\begin{enumerate} \item The first one is obtained by fixing an element $\varepsilon \in \ZZ[t^{\pm 1}]$ and putting
\begin{equation}\label{eqn:Alex2}
f(a):=a +\varepsilon \qquad \forall a \in \ZZ[t^{\pm 1}].
\end{equation}
This virtual quandle is used by Manturov to define the virtual Alexander polynomial, carrying extremely rich topological information about a link. Recall that in the classical setting the Alexander quandle leads to the usual Alexander polynomial, for instance through the Burau representation (cf. example \ref{ex:AlexBurau}).

Note that morphism \eqref{eqn:Alex2} is not linear.
\item
The second virtual structure is more interesting in a slightly generalized context: one replaces $ \ZZ[t^{\pm 1}]$ with an arbitrary $\ZZ[t^{\pm 1}]$-module $A,$ keeping the formula \eqref{eqn:Alex} as a definition of quandle structure. If $A$ is moreover a  $\ZZ[t^{\pm 1},s^{\pm 1}]$-module, then 
\begin{equation}\label{eqn:Alex3}
f(a):=sa \qquad \forall a \in A
\end{equation}
defines a virtual quandle structure, linear this time.
\end{enumerate} 
\end{example}

\medskip
The aim of the rest of this paper is to add some patches to the theory of virtual braids and SD structures, virtualizing (a part of) the content of section \ref{sec:braids}.

\subsection{Patch 1, a small one: positive virtual braid monoids}\label{sec:positive}

Weakening the notion of virtual braid groups by passing to the positive part also makes sense in the virtual world:

\begin{definition}
The \emph{positive virtual braid monoid} $VB_n^+$ is defined by a set of monoid generators $\{\sigma_i,\zeta_i, \;  1 \le i \le n-1\},$ and relations identical to those from definition \ref{def:VB}.
\end{definition}

One gets a submonoid of $VB_n.$

As in the real world, the strucure of shelf bears an action of this monoid:

\begin{proposition}\label{thm:VSaction}
\begin{itemize}
\item Given a shelf $(S,\lhd),$ formulas \eqref{eqn:VBactionOnRacks} and \eqref{eqn:VBactionOnRacks2} define an action of the positive virtual braid monoid $VB_n^+$ on $S^{\times n}.$ 
\item Given a virtual shelf $(S,\lhd,f),$ formulas \eqref{eqn:VBactionOnVRacks} and \eqref{eqn:VBactionOnRacks2} define an action of the positive virtual braid monoid $VB_n^+$ on $S^{\times n}.$ 
\end{itemize}
\end{proposition}

\subsection{Patch 2, an open one: free virtual self-distributive structures}\label{sec:free}

Here we make some steps towards understanding the structure of free virtual SD structures and the action of the $VB_n$'s or $VB_n^+$'s on them, leaving however many open questions. Notation $FVS_n$ will stand for a free virtual shelf on $n$ generators, and similarly for racks and quandles.

\subsubsection{Adding virtual copies of elements}

First, one easily verifies 

\begin{proposition}
The morphism of shelves defined by
\begin{align*}
FS_{\ZZ} &\longrightarrow FVS_1,\\
x_k &\longmapsto f^k(x),
\end{align*}
where $x:=x_1$ is the generator of $FVS_1,$ is an isomorphism. Isomorphisms for racks $FR_{\ZZ} \stackrel{~}{\longrightarrow} FVR_1$  and quandles $FQ_{\ZZ} \stackrel{~}{\longrightarrow} FVQ_1$ can be defined in the same way.
\end{proposition}

In what follows we implicitly use this isomorphism, writing $x_k$ instead of $f^k(x)$ when working in $FVS_1.$

Similarly, $FVS_n$ can be seen as a free shelf with a separate ``virtual" copy $x_{i,k}$ of $x_i$ for all $k \in \ZZ.$

\begin{corollary}
A free virtual SD structure can be seen as a free SD structure (on a larger set of generators).
\end{corollary}

\subsubsection{Free virtual shelves: an attempt to adapt Dehornoy's methods}

Let us now work with shelves, trying to understand how nice the $VB_n^+$-actions on $FS_1^{\times n}$ and $FVS_1^{\times n}$ are (cf. proposition \ref{thm:VSaction}). These actions will be called \emph{real} and \emph{virtual} respectively for brevity. The generator $x_1$ of $FS_1$ will be denoted by $x.$ The following \emph{devirtualization} shelf morphism will be systematically used to extend known results for $FS_1$ to the virtual world:
\begin{align*}
FVS_1 \simeq FS_{\ZZ} & \stackrel{devirt}{\twoheadrightarrow} FS_1,\\
x_k &\longmapsto x.
\end{align*} 

First, the freeness result from proposition \ref{thm:FSActsFreely} does not hold in the virtual context:

\begin{lemma}
The virtual action of the positive virtual braid monoid $VB_n^+$ on $FVS_1^{\times n}$ is not free.
\end{lemma}

\begin{proof}
It is sufficient to notice that all $\zeta_i$'s act as identity on $n$-tuples of the form $(x_k,x_{k+1},\ldots, x_{k+n-1}) \in FVS_1^{\times n}.$
\end{proof}

The author does not know if the  virtual action is faithful. Here are some arguments giving hope for it.

\medskip

Recall that choosing flips as actions corresponding to $\zeta_i$'s can lead to a forbidden YB relation, thus implying non-faithfulness. For the virtual action there is no such danger, as one can easily check

\begin{lemma}
The virtual action distinguishes the two sides of each forbidden YB relation from remark \ref{rmk:OtherYB}.
\end{lemma}

\medskip

A more refined study of the structure of $FVS_1$ is needed to prove further results.

The following definition is inspired by \cite{Dehornoy}.

\begin{definition}
Fix an alphabet $X.$ The \emph{free magma} $T_{X}$ on $X$ is the closure of the set $X$ under the formal (non-associative!) operation $(t_1,t_2)\mapsto t_1 \ast t_2.$ The elements of $T_{X}$ are called \emph{terms}. Notations  $T_{\{x_1, \ldots,x_n \} }$ and $T_{\{x_i, i \in \ZZ \} }$ are abbreviated as $T_n$ and $T_{\ZZ}$ respectively.
\end{definition}

Consider the maps
\begin{align*}
\d: T_{\ZZ} &\twoheadrightarrow FVS_1,& \d: T_{1} &\twoheadrightarrow FS_1,\\
x_i &\mapsto x_i,&x_1 &\mapsto x,\\
\ast &\mapsto \lhd; & \ast &\mapsto \lhd.
\end{align*} 
Concretely, one simply factorizes by relation \eqref{eqn:SelfDistr}. The notation $\d$ comes from the word ``distributivity".

\begin{definition}
Take a term $t=((x_f \ast t_1) \ast \cdots )\ast t_k$ in $T_{\ZZ}$ or $T_n.$ Its \emph{first subscript}, denoted by $\f(t),$ is defined to be $f.$ Further, its \emph{sequence / multiset of first subscripts} is the sequence / multiset formed by $\f(t_1), \ldots,\f(t_k).$ The multiset of first subscripts is denoted by $\of(t).$ Finally, $l(t):=k$ is called the \emph{length} of $t.$
\end{definition}
Note that for $n=1$ only the length function $l$ is relevant.

Playing with relation \eqref{eqn:SelfDistr}, one gets
\begin{lemma}
For two terms giving the same shelf element, i.e. $t, t' \in T_{\ZZ}$ such that $\d(t)=\d(t'),$ one has $l(t)=l(t'),$ $\f(t)=\f(t')$ and $\of(t)=\of(t').$ Moreover, given a $t \in T_{\ZZ}$ with $l(t)=k$ and a $\theta \in S_k,$ there exists a $t' \in T_{\ZZ}$ such that $\d(t)=\d(t')$ and their sequences of first subscripts differ precisely by the permutation $\theta.$ The same holds in $T_1.$
\end{lemma}

Thus one can define $l(a),$ $\f(a)$ and $\of(a)$ for any $a \in FVS_1$ or $a \in FS_1$ as $l(t),$ $\f(t)$ and $\of(t)$ for any term $t$ representing $a.$ 

\medskip

This lemma permits to extract useful information from the virtual action. For this, consider the forgetful monoid morphism
\begin{align*}
\for: VB_n^+ &\twoheadrightarrow S_n,\\
\zeta_i, \sigma_i &\mapsto \zeta_i.
\end{align*}
The $\zeta_i$'s on the right denote standard generators of $S_n.$

\begin{proposition}\label{thm:VActionFor}
\begin{enumerate}
\item The real action of $\theta \in VB_n^+$ on $FS_1^{\times n}$ permits to recover $\for(\theta)$ and the number of the $\sigma_i$'s in $\theta.$ 
\item The virtual action of $\theta \in VB_n^+$ on $FVS_1^{\times n}$ permits to recover $\for(\theta)$ and the number of the $\sigma_i$'s in $\theta.$ 
\item The virtual action of $\theta \in VB_n^+$ on $FVS_1^{\times n}$ permits to recover, for each strand of the virtual braid $\theta,$ the multiset of strands passing (non-virtually!) under it.
\end{enumerate}
\end{proposition}
Note that the multiset in the last point is stable under all authorized YB relations and is thus well-defined. It can be seen as a generalization of linking numbers.

\begin{proof}
\begin{enumerate}
\item Put $(a_1,a_2\ldots, a_n):=\theta(x,x,\ldots,x)$ and $l_j:=l(a_j).$ Changing the $i$th element $x$ in the $n$-tuple $(x,x,\ldots,x)$ to $x \lhd x$ increases exactly one of the $l_j$'s by one. This $j$ is precisely the value of $\for(\theta)(i).$

Further, the real action of a $\zeta_i$ on $FS_1^{\times n}$ does not change the total length of the elements of the $n$-tuple, whereas the real action of a $\sigma_i$ increases it by $1.$ Thus one recovers the number $M$ of the $\sigma_i$'s in $\theta.$ 

\item Follows from the previous point by devirtualizing.

\item
In the virtual context the $\f$'s and $\of$'s refine the information given by the length function. 

The virtual action of a $\zeta \in S_n$ seen as an element of $VB_n^+$ (intuitively this is clear; a rigorous definition will be given later) can be written for ``simple" $n$-tuples like this:
$$\zeta (x_{i_1},x_{i_2}, \ldots)= (x_{i_{\zeta^{-1}(1)}+1-\zeta^{-1}(1)},x_{i_{\zeta^{-1}(2)}+2-\zeta^{-1}(2)}, \ldots).$$ 
Note that $|k-\zeta^{-1}(k)| \le n-1$ for all the $k$'s. In general, working with first subscripts, one sees that $\zeta$ applied to a general $n$-tuple changes the first subscript $\f$ of the element on each strand at most by $n-1.$ Recall the number $M$ of the $\sigma_i$'s in $\theta$ determined in the previous point. Remark also that each $\sigma_i$ simply switches the first subscripts of two of the elements in an $n$-tuple.

Summarizing, the action of our $\theta,$ as well as its subwords, changes the $\f$ of the element on any strand at most by $(n-1)(M+1).$ Put $$N:= (n-1)(M+1)+1$$ and $$(y_1,y_2,\ldots, y_n):=\theta (x_{2N}, x_{4N}, \ldots, x_{2nN}).$$ The $i$th strand of $\theta$ will be called $2iN$ for simplicity.

For any $i,$ replacing each number in $\of(y_i)$ by the closest multiple of $2N,$ one recovers the multiset of strands passing under the strand corresponding to the multiple of $2N$ which is the closest to $\f(y_i).$  This follows from the observations $\of(a \lhd b) = \of(a)\cup \f(b)$ and $\f(a \lhd b) = \f(a),$ from the explicit formulas defining the virtual action, from the estimations for subscript modifications above, and from the independence of $\f$ and $\of$ from the choice of term representing the braid.
\end{enumerate}
\end{proof}

\medskip

Consider now monoid morphisms
\begin{align}\label{eqn:BraidPerm}
S_n & \stackrel{i_S}{\longrightarrow} VB_n^+,  & B_n^+ & \stackrel{i_B}{\longrightarrow} VB_n^+, \\
\zeta_i &\longmapsto \zeta_i;   & \sigma_i &\longmapsto \sigma_i.\notag
\end{align} 

\begin{proposition}\label{thm:SBparts}
The action of $S_n$ (resp. $B_n^+$) on $FS_1^{\times n}$ induced from the real action of $VB_n^+$ via morphism $i_S$ (resp. $i_B$) is faithful (resp. free).
\end{proposition}

\begin{proof}
The statement about $B_n^+$ follows from proposition \ref{thm:FSActsFreely}.

As for $S_n,$ its induced action is the usual action by permutations. Elements $a_k:=((x \lhd x) \lhd \cdots )\lhd x,$ with $k$ occurrences of $x,$ have different lengths ($l(a_k)=k-1$) and are hence pairwise distinct. Therefore a permutation $\zeta \in S_n$ is completely defined by $\zeta(a_1,a_2, \ldots, a_n).$
\end{proof}

Devirtualizing, as usual, one gets the same statement for the virtual action.

Results of this kind allow one to easily get a useful

\begin{corollary}
The submonoid  of $VB_n^+$ generated by the $\zeta_i$'s (resp. $\sigma_i$'s) is isomorphic to $S_n$ (resp. $B_n^+$).
\end{corollary}

\begin{proof}
The submonoids in question are images of $i_S$ and $i_B,$ which are monoid injections according to the preceding proposition.
\end{proof}

\medskip

It is now time for some remarks on a generalization of the Dehornoy order to $FVS_1.$ This partial order is defined by
$$a = c \lhd b \;\Longrightarrow \; b \prec a.
$$
Devirtualizing and using the acyclicity of the Dehornoy order on $FS_1$ (c.f. \cite{Dehornoy}), one sees that only one of the relations $a=b, b \prec a, a \prec b$ can hold for given $a,b \in FVS_1.$ In particular, the order $<$ generated by $\prec$ is acyclic. 

The order $<$ is unfortunately far from being total: the $x_i$'s are all minimal elements (since $b < a$ entails $l(a)>0,$ whereas $l(x_i)=0$) hence mutually incomparable.

The author knows no reasonable total order either on $FVS_1$ or on $VB_n^+.$ Note that one can not hope for a left- or right-invariant order on $VB_n^+$ since it has torsion ($\zeta_i^2=1$).

\medskip

The last result concerns the case $n=2.$

\begin{proposition}
The real action of $VB_2^+$ on $FS_1^{\times 2}$ is faithful.
\end{proposition}

\begin{proof}
Put $\zeta:=\zeta_1, \sigma:=\sigma_1.$
An element $\theta$ of $VB_2^+$ can be uniquely written, after applying $\zeta^2=1$ several times, in its shortest form $\theta= \zeta^{\varepsilon_k}\sigma \cdots  \sigma \zeta^{\varepsilon_1}\sigma \zeta^{\varepsilon_0},$ where $\varepsilon_i \in \{0,1\}.$ 

We first prove that the value of $$(a,b):=\theta(x,x)\in FS_1^{\times 2}$$ allows one to determine whether $k>0$ and, if so, to calculate $\varepsilon_k$ and  $$(a',b'):=\theta'(x,x)\in FS_1^{\times 2},$$ where $\theta = \zeta^{\varepsilon_k}\sigma \theta'.$ Indeed, consider three possibilities and their consequences:
\begin{enumerate}
\item $k=0 \; \Longrightarrow \;(a,b) = (x,x)$;
\item $k>0, \varepsilon_k=0 \; \Longrightarrow \; a \prec b$;
\item $k>0, \varepsilon_k=1 \; \Longrightarrow \; b \prec a.$
\end{enumerate}
Here $\prec$ is the partial Dehornoy order.
The acyclicity of the Dehornoy order proves that only one of the relations $a=b, b \prec a, a \prec b$ can hold. Thus the pair $(a,b)$ tells whether $k>0$ and, if so, calculates $\varepsilon_k.$ To determine $(a',b'),$ recall the right cancellativity of $FS_1$ (i.e. $a \lhd b = a' \lhd b \:\Rightarrow \: a = a',$ c.f. \cite{Dehornoy}). Thus, for instance, relation $a \prec b$ implies that there exists a unique $c \in FS_1$ with $b=c\lhd a,$ so $(a',b') = (c,a).$ Case $b \prec a$ is similar.

Proceeding by induction, one gets the value of $k$ and all the $\varepsilon_i$'s for $i>0.$ This determines $\theta$ up to the rightmost $\zeta.$ To conclude, observe that the presence or absence of this rightmost $\zeta$ determines $\for(\theta)\in S_2 = \{\Id,\zeta\}$ which, according to proposition \ref{thm:VActionFor}, can be read from the real action. 
\end{proof}

In fact we have proved a more precise property: $VB_2^+$ acts freely on couples of the form $(a,b) \in FS_1^{\times 2}$ with distinct $a$ and $b$ which are not directly comparable (i.e. one has neither $b \prec a$ nor $a \prec b$). An example of such $a$ and $b$ is given by $x$ and $(x\lhd x)\lhd x.$

Devirtualizing, one gets
\begin{corollary}
The virtual action of $VB_2^+$ on $FVS_1^{\times 2}$ is faithful.
\end{corollary}

\begin{remark}
The author does not know whether $FVS_1$ is right cancellative. If it were true, the preceding proof could be easily adapted to show that $VB_2^+$ acts freely on couples $(a,b) \in FVS_1^{\times 2}$ which are not directly comparable (i.e. $b \neq f(a)$ and, for all $k \in \ZZ,$ one has neither $b \prec f^k(a)$ nor $a \prec f^k(b)$). An example of such $a$ and $b$ is given by $x_i$ and $x_j$ with $j-i \neq 1.$
\end{remark}

\subsubsection{Free virtual quandles and a conjecture of Manturov}

Let us now turn to free virtual quandles. Developing example \ref{ex:ConjVR}, where $\Conj(F_{n+1})$ was endowed with a virtual quandle structure by taking $f(a):= a \lhd x_{n+1} \; \forall a \in F_{n+1},$ one gets a virtual analogue of the quandle injection \eqref{eqn:FQ}:

\begin{proposition}
The virtual quandle morphism defined on the generators by
\begin{align*}
FVQ_n & {\longrightarrow} \Conj(F_{n+1}),\\
x_i &\longmapsto x_i \qquad \forall 1\le i \le n, \\
\end{align*}
is injective.
\end{proposition}

The image $VConj_n$ of this injection consists of all the conjugates in $\Conj(F_{n+1})$ of the $x_i$'s with $1\le i \le n.$ In particular, $x_{n+1} \in Conj(F_{n+1})$ plays a role different from that of other $x_i$'s: it is not a generator of the virtual quandle $VConj_n,$ but is it here to give the ``virtualizing" morphism $f.$ 

A conjecture raised by Manturov in \cite{ManRecognition} (cf. \cite{ManRecognitionEn} for an English version) is equivalent to the following:

\begin{conj}
The virtual braid group $VB_n$ acts freely on  $(x_1,x_2, \ldots, x_n) \in FVQ_n^{\times n},$ thus generalizing theorem \ref{thm:ActionOnFGn} and  proposition \ref{thm:FQActsFaith}.
\end{conj}

Manturov formulated his conjecture in terms of cosets 
$$E_i:= \{x_i\} \backslash F_{n+1} = F_{n+1} / (a = x_i a \: \forall a).$$
He endowed $E:=\sqcup_{i=1}^{n} E_i$ with the operations
$$a\ast b := ab^{-1}x_j b \in E_i \qquad \forall a \in E_i, b \in E_j,$$
$$f(a)=a x_{n+1} \in E_i \qquad \forall a \in E_i.$$
To see that Manturov's conjecture is equivalent to the one given above, note that $(E, \ast, f)$ is a virtual shelf and that one has a virtual shelf isomorphism
\begin{align*}
E & \stackrel{\sim}{\longrightarrow} VConj_n,\\
a \in E_i &\longmapsto a^{-1}x_i a.
\end{align*}

\subsection{Patch 3, a symmetric one: a categorical counterpart of virtual braids}

Now let us look for a {categorical} counterpart -- in the sense of theorem \ref{thm:AlgCat} -- of the notion of virtual braids. Interesting results in this direction were obtained by Kauffman and Lambropoulou in \cite{KauffLambr}. They introduced the String Category and explored its tight relationship with, on the one hand, the algebraic Yang-Baxter equation, and, on the other hand, virtual braid groups. Morally, passing from the usual to the algebraic YB equation requires, besides a braiding, a (substitute for the) flip, thus suggesting connections with virtual braid groups. This point of view turns out to be quite fruitful, in particular when working with pure braid groups. In this section we present our categorification of $VB_n$ which is quite different from the one from \cite{KauffLambr}. It is closer in spirit to the categorification of $B_n,$ and it produces a handy machine for constructing representations of $VB_n$'s. Examples of such representations will be given in further sections.

Two types of braiding present in the definition of $VB_n$ suggest looking at categories with two distinct braided structures. But this approach is too naive to work: the naturality of the braidings would imply that one can pass any braiding ``through" the other one, meaning one of the forbidden YB relations (cf. remark \ref{rmk:OtherYB}). Thus one needs an adequate non-functorial notion of braiding.

\begin{definition}
An object $V$ in a monoidal category $\C$ is called \emph{braided} if it is endowed with a \emph{specific braiding}, i.e. an invertible morphism
 $$\sigma=\sigma_{V} : V\otimes V \rightarrow V \otimes V,$$
satisfying a categorified version of \eqref{eqn:BrYB}:
\begin{equation}
(\sigma_{V}\otimes \Id_V)\circ(\Id_V \otimes \sigma_{V})\circ(\sigma_{V}\otimes \Id_V) =(Id_V \otimes \sigma_{V}
)\circ(\sigma_{V}\otimes \Id_V)\circ(\Id_V \otimes \sigma_{V}). \label{eqn:YB}\tag{YB}
\end{equation}
\end{definition}

For every object $V$ in a braided category, $c_{V,V}$ is automatically a specific braiding, since the Yang-Baxter equality for $c_{V,V}$ follows from its naturality (take $W=V \otimes V$ and $g=c_{V,V}$ in the condition \eqref{eqn:nat} expressing naturality). But the most interesting case for us is when $V$ has, in addition to $c_{V,V},$ another specific braiding $\sigma_{V}.$ 

\begin{theorem}\label{thm:AlgCatVirtual}
Denote by $\CCbr$ the free symmetric category generated by a single braided object $(V,\sigma_{V}).$ Then for each $n$ one has a group isomorphism
\begin{align}
\psi_{virt}: VB_n &\stackrel{\sim}{\longrightarrow} \End_{\CCbr}(V^{\otimes n})\notag \\
\zeta_i &\longmapsto c_i := \Id_V^{\otimes (i-1)}\otimes c_{V,V} \otimes \Id_V^{\otimes (n-i-1)},\label{eqn:vbr_grp_cat}\\
\sigma_i^{\pm 1} &\longmapsto \sigma_i^{\pm 1} :=\Id_V^{\otimes (i-1)}\otimes \sigma_{V}^{\pm 1} \otimes \Id_V^{\otimes (n-i-1)}.\label{eqn:vbr_grp_cat2}
\end{align}
\end{theorem}

Notation $\CCbr$ emphasizes that two different braidings are present in the story.

\begin{proof}
\begin{enumerate}
\item \label{item1} To check that $\psi_{virt}$ is well defined, one should check three instances of Yang-Baxter equation in $\CCbr,$ the other verifications being trivial. Yang-Baxter equation for $\sigma_{V}$ holds by definition. That for $c_{V,V}$ was proved above. The remaining one, with one occurrence of $\sigma_{V}$ and two of $c_{V,V},$ is a consequence of the naturality of $c$: take $W=V \otimes V$ and $g=\sigma_{V}$ in \eqref{eqn:nat}.
\item \label{item2} To see that the $c_i$'s and $\sigma_i$'s generate the whole $\End_{\CCbr}(V^{\otimes n}),$ remark that the braiding $c$ on tensor powers of $V,$ which is the only part of the structure not described yet, is automatically expressed, due to \eqref{eqn:br_cat} and \eqref{eqn:br_cat2}, via $c_i$'s:
\begin{equation}\label{eqn:br_powers}
c_{V^{\otimes n},V^{\otimes k}}= (c_k\cdots c_1)\cdots(c_{n+k-2}\cdots c_{n-1})(c_{n+k-1} \cdots c_{n}),
\end{equation}
and similarly for $c_{V^{\otimes n},V^{\otimes k}}^{-1}.$ 
\item It remains to show that all relations in $\End_{\CCbr}(V^{\otimes n})$ follow from those which are images by $\psi_{virt}$ of relations from $VB_n$'s.

Equations \eqref{eqn:br_cat} and \eqref{eqn:br_cat2} for $c$ on tensor powers of $V$ are guaranteed by \eqref{eqn:br_powers}. So is the idempotence of $c_{V^{\otimes n},V^{\otimes k}}.$ Naturality of $c_{V^{\otimes n},V^{\otimes k}}$ is the only condition left. According to point \ref{item2}, it suffices to check it for generating morphisms $c_i$'s and $\sigma_i$'s only. But then everything follows from appropriate versions of Yang-Baxter equation discussed in point \ref{item1}.
\end{enumerate}
\end{proof}

Thus virtual braid groups describe hom-sets of a free symmetric category generated by a single braided object.

\begin{corollary}\label{crl:AlgCatVirtual}
For any braided object $V$ in a symmetric category $\C,$ the map defined by formulas \eqref{eqn:vbr_grp_cat} and \eqref{eqn:vbr_grp_cat2} endow $V^{\otimes n}$ with an action of the group $VB_n.$ 
\end{corollary}

Recall the notation \eqref{eqn:phi_i}.
\begin{definition}\label{def:AlgCatVirtual}
Given an object $V$ in a monoidal category $\C,$ we briefly say that a pair  $ (\xi,\vartheta)$ of endomorphisms of $V\otimes V$ \emph{defines a $VB_n$ (or $VB_n^+$) action} if the assignments $\zeta_i \mapsto \xi_i$ and $\sigma_j \mapsto \vartheta_j$ define  an action of $VB_n$ (or $VB_n^+$) on $V^{\otimes n}.$
\end{definition}
The above corollary says for instance that $(c_{V,V},\sigma_V)$ defines a $VB_n$ action.

Thus one has a machine for constructing representations of $VB_n.$ Section \ref{sec:examples} is devoted to concrete examples.

\begin{remark}
Observe that it is somewhat easier to check that one has a symmetric category with a braided object than to verify directly that one actually has an action of $VB_n.$ It comes from the fact that some of the relations in $VB_n$ are already ``built-in" on the categorical level:
\begin{itemize}
\item commutation relations \eqref{eqn:BrComm} and \eqref{eqn:BrCommMixed} are ``hidden" in the definition of the action
$\psi_{virt}$;
\item Yang-Baxter equations \eqref{eqn:BrYBMixed} and \eqref{eqn:BrYB} for $\zeta_i$'s are consequences of the naturality of the braiding $c.$
\end{itemize} 
\end{remark}

\begin{remark}\label{rmk:local}
The idea of working with ``local" braidings on $V$ instead of demanding the whole category $\C$ to be ``globally" braided is similar to what is done in \cite{Goyvaerts}, where self-invertible YB operators are considered in order to define YB-Lie algebras in an additive monoidal category $\C.$ Note also that a specific braiding $\sigma_V$ on $V\in \Ob(\C)$ ``globalizes" to a usual braiding on the monoidal subcategory of $\C$ generated by the object $V$ and the morphism $\sigma_V,$ thanks to a formula analogous to \eqref{eqn:br_powers}. In particular, a free symmetric category generated by a single braided object is isomorphic to $\CCbr$: one needs more structure for the difference between specific and usual ``global" braidings to be perceptible.
\end{remark}

\medskip

Positive virtual braid monoids can be categorified similarly:
\begin{definition}
An object $V$ in a monoidal category $\C$ is called \emph{weakly braided} if it is endowed with a \emph{weak braiding}
 $$\sigma=\sigma_{V} : V\otimes V \rightarrow V \otimes V,$$
satisfying \eqref{eqn:YB}.
\end{definition}

Note that one no longer demands $\sigma_{V}$ to be invertible.

A braided object is automatically weakly braided. Some interesting non-invertible braidings will be presented in section \ref{sec:examples}.
Although very natural, this weak notion of braiding seems to be unexplored.

We finish this section by an obvious generalization of theorem \ref{thm:AlgCatVirtual}.

\begin{theorem}\label{thm:AlgCatVirtualWeak}
Denote by $\CCbr^+$ the free symmetric category generated by a single weakly braided object $(V,\sigma_{V}).$ Then for each $n$ one has a monoid isomorphism
\begin{align}
\psi_{virt, weak}: VB_n^+ &\stackrel{\sim}{\longrightarrow} \End_{\CCbr^+}(V^{\otimes n})\notag \\
\zeta_i &\longmapsto  \Id_V^{\otimes (i-1)}\otimes c_{V,V} \otimes \Id_V^{\otimes (n-i-1)},\notag\\
\sigma_i &\longmapsto \Id_V^{\otimes (i-1)}\otimes \sigma_{V} \otimes \Id_V^{\otimes (n-i-1)}.\notag
\end{align}
\end{theorem}

\subsection{Filling patch 3: ``structural" braidings}\label{sec:examples}

This section is devoted to concrete examples, classical and original, of (weakly) braided objects in symmetric categories. These examples can be considered within the philosophy of our previous paper \cite{Lebed1}. In that paper we constructed, for simple algebraic structures (e.g. an algebra), braidings involving the defining morphisms of these structures (e.g. the multiplication in the case of an algebra). Those braidings encode the structures, i.e. they satisfy \eqref{eqn:YB} if and only if the defining morphisms satisfy the defining relations of these structures (e.g. associativity for an algebra). While the purpose of \cite{Lebed1} was to recover basic homologies of algebraic structures as braided space homologies of the constructed braidings, here we will discover an independent interest of such braidings in the virtual world.

\subsubsection{Shelves and racks revisited}\label{sec:racks2}

Take a set $S$ endowed with a binary operation $\lhd:S\times S \rightarrow S.$ Consider the application 
\begin{align}\label{eqn:RackBraid}
 \sigma = \sigma_{S,\lhd} :  S \times S & \longrightarrow S \times S, \notag\\
 (a,b) & \longmapsto (b,a \lhd b).\tag{RackBraid}
\end{align} 

The following result is well-known:
\begin{lemma}\label{thm:shelf} 
The map $\sigma_{S,\lhd}$ constructed out of a binary operation $\lhd$ on $S$ is:
\begin{enumerate}
\item a weak braiding on $S$ if and only if $(S,\lhd)$ is a shelf;
\item a specific braiding on $S$ if and only if $(S,\lhd)$ is a rack.
\end{enumerate}
\end{lemma}

In other words, the Yang-Baxter equation can be coded by self-distributivity in the algebraic representation theory.

\begin{corollary}
 A rack (or a shelf) $S$ is a (weakly) braided object in the symmetric category $\Set.$ 
\end{corollary}

Thus corollary \ref{crl:AlgCatVirtual} and its non-invertible version give an action of $VB_n$ (resp. $VB_n^+$) on $S^{\times n}.$ This action turns out to be precisely the one from propositions \ref{thm:VBActsOnRacks} (resp. \ref{thm:VSaction}). 

\medskip
We finish with a concrete example.
\begin{example}\label{ex:AlexBurau}
Take the Alexander quandle \eqref{eqn:Alex}. Recall the faithful forgetful functor $For$ from \eqref{eqn:For}. Observe that the braiding $\sigma_{S,\lhd}$ is linear, hence it can be pulled back to $\kVectSum,$ with $\kk = \ZZ[t^{\pm 1}].$
One recovers a virtual version of the Burau representation \eqref{eqn:Burau}, studied in detail by Vershinin in \cite{Ver}. The $\sigma_i$'s act as in the classical case, and the action of ``virtual" $\zeta_i$'s can be written in the matrix form as
$$\rho(\zeta_i)=\begin{pmatrix} I_{i-1} & 0&0&0 \\0&0&1&0\\ 0 &1&0&0\\0&0&0&I_{n-i-1}\end{pmatrix}.$$
\end{example}

\subsubsection{Associative and Leibniz algebras in categories}\label{sec:algebras}

In this subsection we turn to ``structural" braidings for associative / Leibniz algebras constructed in \cite{Lebed1}. Notation $\kk$ will always stand for a commutative ring here.

\medskip
First, recall that a Leibniz algebra is a ``non-commutative Lie algebra". That is,

\begin{definition}
A \emph{Leibniz $\kk$-algebra} is a $\kk$-module $V$ endowed with a bilinear operation $[,] : V\otimes V \longrightarrow V$ satisfying
\begin{equation}\label{eqn:Lei}\tag{Lei}
  [v,[w,u]]=[[v,w],u]-[[v,u],w] \qquad \forall v,w,u \in V.
\end{equation} 
\end{definition}

Remark that one gets the notion of a {Lie algebra} when adding the antisymmetry condition.

Leibniz algebras were discovered by Bloh in 1965, but it was Loday who woke the general interest in this structure in 1989 in his work on lifting the classical Chevalley-Eilenberg complex from the exterior to the tensor algebra (cf. \cite{Cyclic}).

\medskip
Now let us cite categorical versions of associative and Leibniz algebras.

\begin{definition}
\begin{itemize}
 \item A \emph{unital associative algebra} (abbreviated as UAA) in a strict monoidal category $\C$ is an object $V$ together with morphisms $\mu:V\otimes V\rightarrow V$ and $\nu:\II\rightarrow V,$ satisfying associativity condition
  $$\mu\circ(\Id_V\otimes\mu)=\mu\circ(\mu \otimes \Id_V)$$
  and unit condition
    $$\mu\circ(\Id_V\otimes\nu)=\mu\circ(\nu \otimes \Id_V)=\Id_V.$$

 \item A \emph{unital Leibniz algebra} (abbreviated as ULA) in a symmetric preadditive category $\C$ is an object $V$ together with morphisms $[,]:V\otimes V\rightarrow V$ and $\nu:\II\rightarrow V,$ satisfying generalized \emph{Leibniz condition}
 $$[,]\circ(\Id_V\otimes [,])=[,]\circ([,] \otimes \Id_V)-[,]\circ([,] \otimes \Id_V)\circ(\Id_V\otimes c_{V,V}):V^{\otimes 3}\rightarrow V$$
  and \emph{Lie unit} condition
   $$[,]\circ(\Id_V\otimes \nu)=[,]\circ(\nu \otimes \Id_V)=0.$$
\end{itemize}
\end{definition}

See for instance \cite{Baez} and \cite{Majid2} for the definition of algebras in a monoidal category, and \cite{Goyvaerts} for a survey on Lie algebras in a symmetric preadditive category.

Note that to write down the defining condition for a ULA, one needs a preadditive structure (for the sum of morphisms to make sense) and a symmetric braiding $c$ on the underlying category. 

Usual notions of algebras are recovered for $\C = \kVect.$ Super- and color Leibniz algebras (cf. \cite{LeiSuper} and \cite{LeiColor}) correspond to the choices $\C = \kVectGrad$ and $\C = _\Gamma\!\kVect.$ 

\medskip
Everything is now ready for constructing ``structural" braidings from \cite{Lebed1} for UAAs and ULAs.

\begin{theorem}\label{thm:BrForAlgebras}
\begin{enumerate}
\item A unital associative algebra $(V,\mu,\nu)$ in a monoidal category $(\C,\otimes,\II)$ is a weakly braided object in $\C,$ with the weak braiding given by
\begin{equation}\label{eqn:sigmaUAA}
\sigma_{V,\mu} := \nu \otimes \mu : V\otimes V = \II \otimes V \otimes V \rightarrow V \otimes V.
\end{equation} 
\item A unital Leibniz algebra $(V,[,],\nu)$ in a symmetric preadditive category $(\C,\otimes,\II,c)$ is a braided object in $\C,$ with the specific braiding given by
\begin{equation}\label{eqn:sigmaULA}
\sigma_{V,[,]}:= c_{V,V}+ \nu \otimes [,].
\end{equation}
\end{enumerate}
\end{theorem}

A version of the braiding for Lie algebras appears in \cite{Crans}.

In $\kVect$ these braidings can be explicitly written as 
$$\sigma_{V,\cdot}(a \otimes b) = \one \otimes a\cdot b,$$
$$\sigma_{V,[,]}(a \otimes b) = b \otimes a + \one \otimes [a,b],$$
where $\one$ is the (Lie) unit of $V.$

Observe that, contrary to $\sigma_{V,[,]},$ the braiding $\sigma_{V,\mu}$ is highly non-invertible and hence not a specific braiding. Note also that these braidings characterize the corresponding structures: an ``if and only if" type result in the spirit of lemma \ref{thm:shelf} is proved in \cite{Lebed1}.

Combining the theorem with corollary \ref{crl:AlgCatVirtual} and its non-invertible version, one gets

\begin{proposition}
\begin{enumerate}
\item Take a UAA $V$ in a symmetric category $(\C,\otimes,\II,c).$ Then the pair $(c_{V,V},\sigma_{V,\mu})$ defines a $VB_n^+$ action.
 \item Take a ULA $V$ in a symmetric preadditive category $(\C,\otimes,\II,c).$ Then the pair $(c_{V,V},\sigma_{V,[,]})$ defines a $VB_n$ action.
\end{enumerate}
\end{proposition}

It would be interesting to better understand these actions.

\section{Patch 4, a higher level one: arranging the higher level wardrobe}

\subsection{Virtuality as the choice of the ``right" world}\label{sec:CatForVB}

One of the advantages of the categoric vision of virtual braid group actions is an enhanced flexibility. In particular, if one has a braided object $(V,\sigma_V)$ in a symmetric category $(\C,\otimes,\II,c),$ then one can change the associated action of $VB_n$ on $V^{\otimes n}$ by changing either the specific braiding $\sigma_V$ of $V,$ or the symmetric braiding $c$ on $\C$ -- it gives a new action for free. The same is true for weakly braided objects and actions of $VB_n^+.$

A nice illustration of this phenomenon are the ``real" and ``virtual" actions of $VB_n$ on a virtual rack, cf. propositions \ref{thm:VBActsOnRacks} and \ref{thm:VBActsOnVRacks}. More precisely, in what follows we interpret 
``virtualization" of the action as moving to a new symmetric category rather than adding extra structure (the ``virtualization morphism" $f$) to a rack.

The symmetric category we suggest is a particular case of the following general construction.

\begin{theorem}\label{thm:C_V,f} 
Take a symmetric category $(\C,\otimes,\II,c)$ and fix an object $V$ with an automorphism $f \in \Aut_\C(V)$ in it.
\begin{enumerate}
\item A monoidal subcategory of $\C$ can be defined by taking as objects tensor powers $ V^{\otimes n}, n \ge 1 $ and $ V^{\otimes 0}:=\II,$ and as morphisms all morphisms in $\C$ compatible with $f,$ i.e.
$$\Hom^f( V^{\otimes n}, V^{\otimes m}):=\{\varphi \in \Hom_\C( V^{\otimes n}, V^{\otimes m}) | f^{\otimes m} \circ \varphi = \varphi\circ f^{\otimes n}\}.$$
\item This subcategory admits, in addition to $c,$ a new symmetric braiding given on $V \otimes V$ by
$$c^f_{V,V}:=(f^{-1}\otimes f)\circ c_{V,V}.$$
\end{enumerate}
\end{theorem}

Before proving the result, note that, thanks to the naturality of $c,$ one has an alternative expression 
$$c^f_{V,V}=c_{V,V} \circ (f\otimes f^{-1})$$
and, more generally, one can push any occurrence of $f^{\pm 1}$ from one side of $c_{V,V}$ to the other. Remark also that some basic morphisms are automatically in $\Hom^f,$ such as
\begin{itemize}
\item identities $\Id_{ V^{\otimes n}}$;
\item morphisms $f^{\pm 1}$;
\item the original braiding $c$ (thanks, as usual, to its naturality).
\end{itemize}

\begin{proof}
\begin{enumerate}
\item One easily checks that all the $\Id_{ V^{\otimes n}}$'s are in $\Hom^f,$ and that the latter is stable by composition and tensor product.
\item First, the previous point and remarks preceding the proof guarantee that $c^f_{V,V} \in \Hom^f.$ Next, extend $c^f$ to other powers $ V^{\otimes n}$ by formula \eqref{eqn:br_powers}. This extension remains in $\Hom^f.$ Such an extension ensures relations \eqref{eqn:br_cat} and \eqref{eqn:br_cat2} for $c^f.$ Remaining properties \eqref{eqn:nat} and \eqref{eqn:symm_cat} for $c^f$ follow from the corresponding properties for $c$ by pushing all the instances of $f^{\pm 1}$ on the left of each expression, using the naturality of $c$ and the compatibility of all morphisms in $\Hom^f$ with $f.$
\end{enumerate}
\end{proof}

\begin{definition}
The symmetric category constructed in the previous theorem 
is denoted by $\C_{V,f}.$
\end{definition}

Now take $\C = \Set$ and let $S$ be a rack. We have seen that $\sigma_{S,\lhd}$ is a specific braiding on $S.$ Given an automorphism $f$ of $S,$ one checks that $\sigma_{S,\lhd}$ is a morphism in the subcategory $\Set_{S,f}$ if and only if $f$ is a rack morphism. Thus, if $(S,f)$ is a virtual rack, then $(S,\sigma_{S,\lhd})$ is still a braided object in the symmetric subcategory $(\Set_{S,f}; \tau)$ of $\Set.$ 
According to the previous theorem, one can modify the symmetric braiding of $\Set_{S,f}$:

\begin{proposition}
The action of $VB_n$ on the braided object $(S,\sigma_{S,\lhd})$ of the symmetric category $(\Set_{S,f}; \tau^f)$ is precisely the one given in proposition \ref{thm:VBActsOnVRacks}.
\end{proposition}

\subsection{Virtually twisted braidings}\label{sec:twisted}

Another advantage of disposing of two different braidings is the possibility of twisting the specific one using the underlying symmetric one:

\begin{theorem}\label{thm:TwistedBr}
Take a (weakly) braided object $(V,\sigma_V)$ in a symmetric category $(\C,\otimes,\II,c).$ Then $V$ can be endowed with another specific (resp. weak) braiding $\sigma'_V:=c_{V,V}\circ\sigma_V\circ c_{V,V}.$ 
\end{theorem}

\begin{proof}
The only non-trivial property to check is equation \eqref{eqn:YB} for $\sigma'_V.$ We treat only specific braidings here, the weak case being similar.

Consider the twisting application 
\begin{align*}
t:VB_n &\longrightarrow VB_n,\\
\theta & \longmapsto \Delta_n \theta\Delta _n,
\end{align*}
where $\Delta_n$ is the Garside element, i.e. the total twist
$\Delta_n:=\bigl( \begin{smallmatrix}
 1 & 2 & \cdots & n\\ n & n-1 & \cdots & 1
\end{smallmatrix} \bigr)
\in S_n,$ seen as an element of $VB_n$ via inclusion \eqref{eqn:BraidPerm}. Since $\Delta_n\Delta_n=1,$ the map $t$ is extremely nice:
\begin{enumerate}
\item $t$ is a group map;
\item $t$ is involutive, hence an isomorphism.
\end{enumerate}
On the generators, $t$ gives
\begin{align*}
t(\sigma_i) &=\sigma'_{n-i},\\
t(\zeta_i) &=\zeta_{n-i},
\end{align*}
where $\sigma':=\zeta\sigma\zeta.$
Equation \eqref{eqn:BrYB} for $\sigma'$ is now a consequence of \eqref{eqn:BrYB} for $\sigma.$

To conclude, notice that $\sigma'_V$ is precisely the action, according to corollary \ref{crl:AlgCatVirtual}, of the element $\sigma'$ of $VB_n.$
\end{proof}

\begin{definition}
We call the braiding from the previous theorem \emph{twisted}.
\end{definition}

The element $\sigma'$ of $VB_n$ (or $VB_n^+$) is graphically depicted as
 
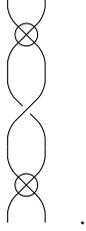
\begin{center}
\begin{tikzpicture}[scale=0.5]
\draw [rounded corners] (5,-2) -- (5,-1.5) -- (4,-0.5) --(4,0);
\draw [rounded corners] (4,-2) -- (4,-1.5) -- (5,-0.5) -- (5,0);
\draw (4.5,-1) circle (0.3);
\draw [rounded corners] (5,0) -- (5,0.5) -- (4.6,0.9);
\draw [rounded corners] (4.4,1.1) -- (4,1.5) --(4,2);
\draw [rounded corners] (4,0) -- (4,0.5) -- (5,1.5) -- (5,2);
\draw [rounded corners] (5,2) -- (5,2.5) -- (4,3.5) --(4,4);
\draw [rounded corners] (4,2) -- (4,2.5) -- (5,3.5) -- (5,4);
\draw (4.5,3) circle (0.3);
\node  at (6,-2){.};
\end{tikzpicture}
\end{center}

This element is quite famous in the virtual knot theory, since many invariants do not distinguish it from the original braiding $\sigma.$

\medskip
In section \ref{sec:CatForVB} we encountered a category with two distinct symmetric braidings. With this in mind, one can state a stronger version of the previous theorem, based on similar observations:

\begin{proposition}\label{thm:TwistedBr2} 
Take a braided object $(V,\sigma_V)$ in a category $(\C,\otimes,\II)$ admitting two symmetric braidings $b$ and $c.$ Put $\sigma''_V:=c_{V,V}\circ b_{V,V}\circ\sigma_V\circ b_{V,V}\circ c_{V,V}$ and $b'_V:=c_{V,V}\circ b_{V,V} \circ c_{V,V}.$   
Then the pair $(b'_V,\sigma''_V)$ defines a $VB_n$ action, isomorphic to the action given by $(b_{V,V}, \sigma_V).$
\end{proposition}

\begin{proof}
Put $\sigma'_V:=b_{V,V}\circ\sigma_V\circ b_{V,V}.$

The involutive action of $\Delta_n:= \bigl( \begin{smallmatrix}
 1 & 2 & \cdots & n\\ n & n-1 & \cdots & 1
\end{smallmatrix} \bigr)
\in S_n$ on $V^{\otimes n}$ via the symmetric braiding $b$ intertwines $(\sigma_V)_i$ and $(\sigma'_V)_{n-i},$ as well as $(b_{V,V})_i$ and $(b_{V,V})_{n-i}.$ Further, the involutive action of $\Delta_n$ on $V^{\otimes n}$ via the second symmetric braiding $c$  intertwines $(\sigma'_V)_i$ and $(\sigma''_V)_{n-i},$ as well as $(b_{V,V})_i$ and $(b'_{V,V})_{n-i}.$ Their composition yields the announced isomorphism.
\end{proof}

An analogue of this result for weak actions can easily be formulated.

Let us now see what this proposition gives in the setting of theorem \ref{thm:C_V,f}, 
which is a source of two symmetric braidings coexisting in a category.
Taking $b=c^f,$ one gets

\begin{proposition}\label{thm:TwistedBr3}
In a symmetric category $(\C,\otimes,\II,c),$ take a braided object $(V,\sigma_V)$ endowed with an automorphism $f$  compatible with the braiding, i.e.
$$\sigma_V \circ (f\otimes f) = (f\otimes f) \circ \sigma_V.
$$
Then the pairs $(c_{V,V}^f,\sigma_V )$ and $(c_{V,V}^{f^{-1}},(f\otimes f^{-1})\circ\sigma_V\circ (f^{-1}\otimes f) )$ give $VB_n$ actions, which are isomorphic.
\end{proposition}

Applying this proposition to categories $\C_{V,f^k},$ one obtains:
\begin{corollary}\label{crl:TwistedBrK}
In the settings of the preceding proposition, the pairs $(c_{V,V}^{f^{k+1}}, \sigma_V)$ and $(c_{V,V}^{f^{k-1}},(f\otimes f^{-1})\circ\sigma_V\circ (f^{-1}\otimes f))$ give isomorphic $VB_n$ actions for any $k \in \ZZ.$
\end{corollary}

\begin{example}
Consider the second virtual quandle structure from example \ref{ex:AlexVirt}. The automorphism $f(a)=sa$ is  compatible with the braiding $\sigma_{A,\lhd}$ since $f$ is a quandle morphism. Then the preceding corollary establishes an isomorphism between $VB_n$ actions given by the pairs $(\tau^{f^{k+1}},\sigma_{A,\lhd})$ and $(\tau^{f^{k-1}},\sigma''_{A,\lhd}),$ where
$$\sigma''_{A,\lhd}(a,b)= (s^2b, ts^{-2}a +(1-t)b).$$
Note that $A$ is also a $\ZZ[u^{\pm 1}, v^{\pm 1}]$-module, with $u$ acting by $s^2$ and $v$ by $ts^{-2}.$ The matrix form of $\sigma'_{A,\lhd}$ is then
$$\begin{pmatrix} 0 & u\\ v &1 -uv\end{pmatrix},$$
 which is precisely the \emph{twisted Burau matrix} (cf. \cite{SW}, or \cite{KaufRad}, where it is recovered via Alexander biquandles).

Further, the isomorphism of actions for $k=1$ can be interpreted, in this example, as follows: virtualizing the $VB_n$ action on a rack, in the sense of proposition \ref{thm:VBActsOnVRacks} (for $f(a)=ua$) is equivalent to twisting the braiding $\sigma_{A,\lhd}$, in the sense of proposition \ref{thm:TwistedBr2}. 
This was noticed in \cite{BartFenn}.
\end{example}

\begin{remark}
One more result from \cite{BartFenn} admits a natural interpretation using the tools developed here. It is the possibility to transform a matrix solution $\bigl( \begin{smallmatrix}
A & B\\ C &D\end{smallmatrix} \bigr)$ of the Yang-Baxter equation into a solution $\bigl( \begin{smallmatrix}
A & tB\\ t^{-1}C &D\end{smallmatrix} \bigr),$ and the equivalence of the two induced representations of the braid group $B_n.$ This, as well as their theorem 7.1, follows from corollary \ref{crl:TwistedBrK} by taking $\C = \kVectSum,$ $k=1$ and $f(v)=sv,$ with $s^2=t.$
\end{remark}

\subsection{Shelves and racks in other worlds}\label{sec:CatSD}

As we have seen in section \ref{sec:algebras}, associative and Leibniz algebras in monoidal categories are defined very naturally and provide a rich source of braided objects in symmetric categories. Since shelves and quandles are braided objects par excellence, it would be interesting to categorify them and to look for new examples of braidings emerging in this generalized setting. Such a categorification will be given here, with several examples in the next section, including -- quite unexpectedly! -- associative and Leibniz algebras.

The main difficulty resides in interpreting the \emph{diagonal map} 
\begin{equation}\label{eqn:diag}
\D_S: a\mapsto (a,a),
\end{equation}
implicit on the right side of \eqref{eqn:SelfDistr}. The flip, equally implicit in \eqref{eqn:SelfDistr} (moving one of the $c$'s before $b$), is also to be treated appropriately. Two approaches are proposed in \cite{CatSelfDistr} (see also \cite{Crans}):
\begin{enumerate}
\item generally, one can work in an additive category, which admits binary products and hence diagonal and transposition morphisms;
\item a concrete example of an additive category is $\mathbf{Coalg},$ the category of counital coalgebras over a field $\k,$ with the comultiplication as diagonal map and the flip as transposition map. 
\end{enumerate}

The approach presented here is quite different. Instead of requiring a diagonal map on the categoric level, we make it part of the categoric SD structure and impose compatibility with the ``product" $\lhd,$ in the spirit of bialgebras. Thus we reserve the underlying symmetric category structure only for the ``virtual part" of $VB_n$ (the $\zeta_i$'s), and the braided object structure (coming, according to lemma \ref{thm:shelf}, from the rack structure) for the ``real part" ($\sigma_i$'s). This seems to us more consistent with the topological interpretation, where virtual crossings are just artefacts of depicting a diagram in the plane, while usual crossings come from the intrinsic knot structure. 

Here is a list of other advantages of our approach: 
\begin{itemize}
\item we work in a general monoidal rather than $\k$-linear setting;
\item no counit is demanded (note that counits pose some problems in \cite{CatSelfDistr}, and they do not exist in one of the examples given below);
\item cocommutativity, often necessary in \cite{CatSelfDistr}, is replaced with a weaker notion - again, with an example when the difference is essential;
\item the flexibility in the choice of the underlying symmetric category allows to treat, among other structures, Leibniz superalgebras.
\end{itemize}
On the negative side, our definition is quite heavy, since, for example, one has to replace conditions like ``a morphism in $\mathbf{Coalg}$" with their concrete meaning. The reader is advised to draw pictures, in the spirit of \cite{CatSelfDistr}, to better manipulate all the notions.

\medskip
Recall notations \eqref{eqn:phi_i} and \eqref{eqn:phi^i}.
\begin{definition}\label{def:shelf_gen}
Take a symmetric category $(\C,\otimes,\II,c).$ A $V\in\Ob(\C)$ is called \emph{a shelf} in $\C$ if it is endowed with two morphisms $\Delta:V\longrightarrow V\otimes V$ and $\lhd:V\otimes V\longrightarrow V$ satisfying the following conditions (where the braiding $c_{V,V}$ is denoted simply by $c$ for succinctness):
\begin{enumerate}
\item $\Delta$ is a coassociative \emph{weakly cocommutative} comultiplication, i.e.
$$\Delta_1\circ \Delta = \Delta_2\circ \Delta:V\longrightarrow V^{\otimes 3},$$
$$c_2\circ \Delta^3 = \Delta^3:V\longrightarrow V^{\otimes 4};$$
\item $\lhd$ is \emph{self-distributive in the generalized sense} (abbreviated as GSD):
$$\lhd^2 = \lhd \circ (\lhd\otimes\lhd) \circ c_2\circ  \Delta_3:V^{\otimes 3} \longrightarrow V;$$
\item the two morphisms are \emph{compatible in the braided bialgebra sense}:
\begin{equation}\label{eqn:BialgCompat}
\Delta \circ \lhd = (\lhd\otimes\lhd) \circ c_2\circ (\Delta\otimes\Delta):V^{\otimes 2}\longrightarrow V^{\otimes 2}.
\end{equation}
\end{enumerate}
A shelf $V$ is called \emph{a rack} in $\C$ if moreover it is endowed with
\begin{enumerate}
\item a right counit $ \varepsilon :V\longrightarrow \II,$ i.e.
$$\varepsilon_2\circ \Delta = \Id_V:V\longrightarrow V,$$
\item a morphism $\wlhd:V\otimes V\longrightarrow V$ which is the ``twisted inverse" of $\lhd$:
$$\wlhd \circ \lhd_1 \circ c_2\circ \Delta_2  = \lhd \circ \wlhd_1\circ c_2\circ \Delta_2 =\Id_V\otimes\varepsilon :V^{\otimes 2}\longrightarrow V.$$
\end{enumerate}
The described structures are called \emph{GSD structures} for brevity.
\end{definition}

Note that the usual cocommutativity implies the weak one; we prefer keeping our weaker condition in order to allow non-cocommutative examples in the next section. 

Our definition is designed for a generalized version of lemma \ref{thm:shelf} to hold:

\begin{proposition}\label{thm:shelf_gen}
A shelf $(V,\Delta,\lhd)$ in a symmetric category $(\C,\otimes,\II,c)$ admits a weak braiding
$$\sigma=\sigma_{V,\Delta,\lhd}:=\lhd_2\circ c_1 \circ \Delta_2.$$
This braiding is specific if $V$ is moreover a rack, the inverse given by 
$$\sigma^{-1}= \wlhd_1 \circ c_2\circ c_1\circ c_2\circ \Delta_1.$$
\end{proposition}

The verifications are easy but lengthy, so they are not given here. Diagrammatic proof is probably the least tiresome. Here is for instance the graphical form of the braiding:
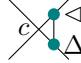
\begin{center}
\begin{tikzpicture}[scale=0.8]
 \draw (1,0) -- (0,1);
 \draw (0,0) -- (1,1);
 \draw (0.75,0.25) -- (0.75,0.75);
 \node at (0.75,0.25) [right] {$\Delta$};
 \node at (0.75,0.75) [right] {$\lhd$};
 \node at (0.5,0.5) [left] {$c$}; 
 \fill[teal] (0.75,0.25) circle (0.1);
 \fill[teal] (0.75,0.75) circle (0.1);
\end{tikzpicture}
\end{center}

\begin{corollary}
In the settings of the previous proposition, the pair $(c_{V,V},\sigma_{V,\Delta,\lhd})$ gives a $VB_n^+$ or $VB_n$ action.
\end{corollary}

\begin{remark}
One could start with proposition \ref{thm:shelf_gen} and ask oneself what conditions on morphisms $\Delta$ and $\lhd$ ensure that $\sigma_{V,\Delta,\lhd}$ is a braiding. In fact, the conditions from definition \ref{def:shelf_gen} are very far from being unique, unlike the ``if and only if" results from previous sections! We cite just two more of multiple examples here.
\begin{enumerate}
\item The weak cocommutativity can be transformed to
 $$c_2\circ c_3\circ \Delta^3 = \Delta^3:V\longrightarrow V^{\otimes 4},$$
and the GSD condition to
$$\lhd^2 = \lhd^2 \circ  c_2\circ \lhd_2 \circ  \Delta_3:V^{\otimes 3} \longrightarrow V.$$
Note that in the cocommutative case this coincides with the original definition.
\item The GSD condition can be substituted with the usual associativity, and the compatibility condition can be made \emph{Yetter-Drinfeld-like}: 
$$\lhd_2 \circ  c_2\circ \Delta_1 \circ \lhd_1 \circ  c_2 \circ\Delta_2 = (\lhd\otimes\lhd) \circ c_2\circ (\Delta\otimes\Delta):V^{\otimes 2}\longrightarrow V^{\otimes 2}.$$
\end{enumerate}
Morally, starting with a bialgebra structure, one should substitute either the compatibility condition with a ``Yetter-Drinfeld-like", or the associativity condition with the GSD one.

Observe that one could also modify the braiding from proposition \ref{thm:shelf_gen} by putting
$$\sigma:=\sigma'_{V,\Delta,\lhd}:=c \circ \lhd_1\circ \Delta_2,$$
which is different from $\sigma_{V,\Delta,\lhd}$ only in the non-cocommutative case.
Conditions similar to those for $\sigma_{V,\Delta,\lhd}$ guarantee that it is a braiding. This braiding makes the rack case less ``twisted".

The choices in definition \ref{def:shelf_gen} are motivated by the examples from the next section.
\end{remark}

\begin{remark}
The generalized self-distributivity can be efficiently expressed with the help of $\sigma$:
\begin{equation}
 \lhd^2 = \lhd^2 \circ \sigma_2.
\end{equation}
In other words, the right ``action" of $V$ on itself via $\lhd$ is ``$\sigma$-commutative".
\end{remark}

\medskip
Now let us move on to examples. The first one is naturally that of usual SD structures. Choose $\Set$ as the underlying symmetric category. Recall the diagonalization map \eqref{eqn:diag}. Further, for a set $S,$ denote by $\varepsilon_S$ the map from $S$ to $\II,$ unique since the one-element set $\II$ is a final object. One easily sees that $(S,\D_{S},\varepsilon_{S})$ is a counital cocommutative coalgebra in $\Set.$ This ensures some of the properties of definition \ref{def:shelf_gen}. Analyzing the remaining ones, one gets:

\begin{proposition}\label{thm:shelf_gen_usual}
\begin{enumerate}
\item The triple $(S,\D_{S},\lhd)$ is a shelf in the symmetric category $\Set$ if and only if $(S,\lhd)$ is a usual shelf.
\item The datum $(S,\D_{S},\varepsilon_{S},\lhd,\wlhd)$ is a rack in the symmetric category $\Set$ if and only if $(S,\lhd,\wlhd)$ is a usual rack.
\end{enumerate}
Moreover, for a shelf $({S},\lhd),$ the braiding $\sigma_{S,\D_{S},\lhd}$ from proposition \ref{thm:shelf_gen} coincides with $\sigma_{S,\lhd}$ from lemma \ref{thm:shelf}.
\end{proposition}

Thus generalized self-distributivity includes the usual one. Examples from the next section show that the generalized notion is in fact much wider.

\subsection{Associative, Leibniz and Hopf algebras under the guise of shelves}	

Start with UAAs. The following result allows to see \textbf{associativity as a particular case of generalized self-distributivity}, for a suitably chosen comultiplication.

\begin{proposition}\label{thm:UAAasShelf}
Take an object $V$ in a monoidal category $(\C,\otimes,\II),$ equipped with two morphisms $\mu:V\otimes V\longrightarrow V$ and $\nu: \II \longrightarrow V,$ with $\nu$ being a right unit for $\mu$:
$$\mu \circ \nu_2 = \Id_V.$$
Put
$$\Delta_{Ass}:=\nu \otimes \Id_V.$$
Then the triple $(V,\Delta_{Ass},\mu)$ satisfies all the conditions from definition \ref{def:shelf_gen} but the GSD, which is equivalent to the associativity of $\mu.$

Moreover, for a UAA $(V,\mu,\nu),$ the braiding $\sigma_{V,\Delta_{Ass},\mu}$ from proposition \ref{thm:shelf_gen} coincides with $\sigma_{V,\mu}$ from theorem \ref{thm:BrForAlgebras}.
\end{proposition}

In $\kVect,$ the comultiplication $\Delta_{Ass}$ becomes simply
$$\Delta_{Ass}(v)=\one \otimes v \qquad \forall v \in V.$$

This example is somewhat exotic. It explains why we were quite demanding in choosing the conditions in definition \ref{def:shelf_gen}. In particular,
\begin{itemize}
\item $\Delta_{Ass}$ is cocommutative in the weak but not in the usual sense;
\item $\Delta_{Ass}$ admits only left counits in general;
\item $(V,\Delta_{Ass},\mu)$ is not a rack in general;
\item the braiding $\sigma_{\Delta_{Ass},\mu}$ is not invertible in general.
\end{itemize}

\medskip
The case of ULAs is somewhat trickier. A natural candidate for comultiplication is 
$$\Delta=\nu \otimes \Id_V+ \Id_V\otimes \nu,$$
but to recover the Leibniz condition \eqref{eqn:Lei} as the GSD one, one wants the ``right multiplication by one" (i.e. $[,]\circ \nu_2: V\longrightarrow V$) to be identity and not zero, as the definition of a ULA imposes. A standard solution is to start with a (not necessarily unital) Leibniz algebra $V$ and to introduce a ``formal unit", i.e. to work in $V\oplus \II.$ One needs an additive category for the direct sum to be defined. To simplify notations, we work in $\kVect$ here. This ``unit problem" turns out to be the only obstacle to interpreting Leibniz algebras via GSD, as the next result witnesses.

Take an $\kk$-module $V$ equipped with a bilinear operation $[,].$ Put $\oV:=V\oplus \kk\one.$ Define a comultiplication $\Delta_{Lei}$ and a counit $\varepsilon_{Lei}$ on $\oV$ by
\begin{align*}
\Delta_{Lei}(v)&:= v\otimes \one + \one\otimes v, & \varepsilon_{Lei}(v)&:=0 \qquad \forall v\in V,\\
\Delta_{Lei}(\one)&:= \one \otimes \one, & \varepsilon_{Lei}(\one)&:= 1,
\end{align*}
 and binary operations $\lhd_{Lei},\wlhd_{Lei}$ on $\oV$ by
\begin{align*}
v\lhd_{Lei} w=-v\wlhd_{Lei} w&:= [v,w] \qquad \forall v,w\in V,\\
v\lhd_{Lei} \one = v\wlhd_{Lei} \one&:= v \qquad \qquad \forall v\in \oV,\\
\one \lhd_{Lei} v=\one\wlhd_{Lei} v &:= 0 \qquad \qquad \forall v\in V.
\end{align*}

\begin{proposition}\label{thm:LAasShelf}

The datum $(\oV,\Delta_{Lei},\varepsilon_{Lei},\lhd_{Lei},\wlhd_{Lei})$ satisfies all the conditions from definition \ref{def:shelf_gen} but the GSD, which is equivalent to the Leibniz condition for $[,].$

Moreover, for a Leibniz algebra $(V,[,]),$ the braiding $\sigma_{V,\Delta_{Lei},\lhd_{Lei}}$ on $\oV$ from proposition \ref{thm:shelf_gen} coincides with $\sigma_{\oV,[,]}$ from theorem \ref{thm:BrForAlgebras}, where $[,]$ is extended to $\oV$ by declaring $\one$ a Lie unit.
\end{proposition}

The GSD structure found here turns out to be the same as in \cite{CatSelfDistr}.

Note that the map $\lhd_{Lei}$ is neither Leibniz nor anti-symmetric due to its definition for $\one.$ The advantage of working with ULAs in section \ref{sec:algebras} was that one always stayed within the Leibniz world. Another advantage was a simple compact formula $\sigma_{V,[,]}= c+ \nu \otimes [,]$ for the braiding, whereas its analogue here $\sigma_{\Delta_{Lei},\lhd_{Lei}}= c+ \nu \otimes \lhd_{Lei}$ is false for elements of the form $v \otimes \one.$ 

Working in other additive categories ($\kVectGrad$ for example), one treats the case of Leibniz super- and color algebras, as well as other structures. 

\medskip

\begin{remark}
In the two preceding examples, it is the particular choice of the comultiplication that dictated the nature of the multiplicative structure (associative or Leibniz algebra). For usual shelves this ``control" is even stronger. Concretely, take the linearization $\kk S$ of a set $S,$ where $\kk$ is a commutative ring without zero divisors. Consider the comultiplication on $\kk S$ which is the linearization of $\D_S$ on $S.$ Suppose that, together with a multiplication $\lhd,$ it endows $\kk S$ with a shelf structure in $\kVect.$ From the compatibility condition \eqref{eqn:BialgCompat} one easily deduces that, for $a,b\in S,$ their product $a \lhd b$ is either zero or an element of $S.$ Thus $\lhd$ ``almost comes from a shelf structure on $S$". This is a generalization of lemma 3.8 from \cite{CatSelfDistr}.

Another example - that of the \emph{trigonometric coalgebra} $T= \CC a \oplus \CC b$ with
\begin{align*}
\Delta(a)&=a\otimes a - b\otimes b,\\
\Delta(b)&=a\otimes b + b\otimes a,
\end{align*}
was considered in \cite{CatSelfDistr}. But the elements $x=a+\imath b, y=a-\imath b$ being group-like (i.e. $\Delta(x)=x \otimes x,$ $\Delta(y)=y \otimes y$), all the GSD structures with trigonometric $\Delta$ are isomorphic to GSD structures with linearized diagonal $\Delta.$ In particular, lemma 3.9 from \cite{CatSelfDistr} is just a reformulation of their lemma 3.8.
\end{remark}

\medskip
The last example, also studied in \cite{CatSelfDistr}, is that of a Hopf algebra.

\begin{proposition}\label{thm:HAasShelf}
Let $(H,\mu,\Delta,\nu,\varepsilon,S)$ be a cocommutative Hopf algebra in a symmetric category $(\C,\otimes,\II,c).$ Define
$$\lhd_H = \mu^2\circ S_1 \circ c_1 \circ \Delta_2: H\otimes H\longrightarrow H,$$
$$\wlhd_H =\lhd_H\circ S^{-1}_2 : H\otimes H\longrightarrow H.$$
The datum $(H,\Delta,\varepsilon,\lhd_{H},\wlhd_{H})$ satisfies all the conditions from definition \ref{def:shelf_gen}. Proposition \ref{thm:shelf_gen} then endows $H$ with a specific braiding.
\end{proposition}

In $\kVect,$ the definition of $\lhd_H$ is written, using Sweedler's notation, as
$$x \lhd_H y = S(y_{(1)})x y_{(2)},$$
$$x \wlhd_H y = y_{(2)}x S^{-1}(y_{(1)}),$$
which are the well-known adjoint maps. The braiding becomes
$$\sigma_{V,\Delta,\lhd_{H}}(x \otimes y) =y_{(1)}\otimes S(y_{(2)})x y_{(3)}.$$
This is precisely the braiding obtained by viewing $H$ as a Yetter-Drinfeld module over itself, cf. \cite{Wor}. 
It shows in particular that the cocommutativity condition is redundant, since it is not used in the Yetter-Drinfeld approach.

\section{A connection to the real world: (co)homologies}\label{sec:hom}

The aim of this section is to show that the existence of braidings described above is more than a mere coincidence. Namely, we recall here the (co)homologies of braided objects (cf. \cite{Lebed1}), which give the familiar Koszul, rack, bar and Leibniz differentials in suitable settings. That theory is applied to GSD structures, braided via proposition \ref{thm:shelf_gen}. An interpretation in terms of (pre)simplicial objects is given. The reader is sent to Przytycki's papers (\cite{Prz1}) for a detailed study of such (pre)simplicial objects for usual SD structures and for concrete homology calculations, as well as for some evidence concerning the relation between homology theories of associative algebras and those of SD structures.

Note that another version of cohomologies of GSD structures was proposed in \cite{CatSelfDistr}. Their definition was inspired by the bialgebra cohomology and extension-deformation-obstruction ideas. The approach developed here is different. Our motivation is a direct generalization of rack and Chevalley-Eilenberg homologies, with potential applications to topology.

\subsection{Braiding + cut + comultiplication = weakly simplicial structure}

Start with recalling the main result of \cite{Lebed1}:
\begin{proposition}\label{thm:BraidedHom}
Let $(\C,\otimes,\II)$ be a preadditive monoidal category. For any weakly braided object $(V,\sigma_V)$ with an \emph{upper cut} $\epsilon,$ i.e. a morphism $\epsilon:V\rightarrow \II$  compatible with the braiding:
$$(\epsilon\otimes\epsilon)\circ \sigma_V = \epsilon\otimes\epsilon,$$
the families of morphisms 
$$({^\epsilon}\! d)_n := \sum_{i=1}^{n}(-1)^{i-1} \epsilon_1 \circ (\sigma_{V^{i-1},V}\otimes \Id_V^{n-i}) :V^{n}\rightarrow V^{n-1},$$
$$(d{^\epsilon})_n := \sum_{i=1}^{n}(-1)^{n-i} \epsilon_n \circ (\Id_V^{n-i}\otimes \sigma_{V,V^{i-1}}) :V^{n}\rightarrow V^{n-1},$$
where $\sigma$ is extended to other powers of $V$ by \eqref{eqn:br_powers},
define a \emph{bidifferential} for $V,$ in the sense that 
$$({^\epsilon}\! d)_{n-1}\circ({^\epsilon}\! d)_n = ({^\epsilon}\! d)_{n-1}\circ (d{^\epsilon})_n + (d{^\epsilon})_{n-1}\circ ({^\epsilon}\! d)_n = (d{^\epsilon})_{n-1}\circ (d{^\epsilon})_n= 0. $$
\end{proposition}

In \cite{Lebed1} these bidifferentials are interpreted in terms of quantum shuffles.

\begin{remark}
In \cite{HomologyYB}, a homology theory was developed for solutions $(S,\sigma)$ of the set-theoretic Yang-Baxter equation using combinatorial and geometric methods completely different from ours. Applications to virtual knot invariants were given. It can be checked that their differential on $S^{\times n}$ coincides with our ${^\epsilon}\!d - d{^\epsilon},$ where the underlying symmetric category is $\Set,$ where $\epsilon$ is the unique map from $S$ to $\II,$ and where the linearization functor $Lin: \Set \rightarrow \mathbf{Mod_\ZZ}$ is used in order to work in a preadditive category.
\end{remark}

\medskip
Here we need a stronger version of this result. First, recall the notion of simplicial objects in a category (cf. \cite{Cyclic} for details and \cite{Prz1} for weak versions; note that our definition is a shifted version of theirs, and our definition of bisimplicial objects is different from the usual one): 

\begin{definition}
Take a category $\C.$ Consider a family $C_n,\: n\ge 0$ of objects in $\C,$ equipped with morphisms $d_{n;i}:C_n\rightarrow C_{n-1}$ (and $d'_{n;i}:C_n\rightarrow C_{n-1}$ and/or $s_{n;i}:C_n\rightarrow C_{n+1}$ when necessary) with $1 \le i \le n,$ denoted simply by $d_i,d'_i,s_i$ when the subscript $n$ is clear from the context. This datum is called
\begin{itemize}
\item \emph{a presimplicial object} if
\begin{align}
d_{i} d_{j} &= d_{j-1} d_{i} & \forall 1\le i < j \le n;\label{eqn:simpl1}
\end{align}
\item \emph{a very weakly simplicial object} if moreover
\begin{align}
s_{i} s_{j} &= s_{j+1} s_{i} & \forall 1\le i \le j \le n,\label{eqn:simpl2}\\
d_{i} s_{j} &= s_{j-1} d_{i} & \forall 1\le i < j \le n,\label{eqn:simpl3}\\
d_{i} s_{j} &= s_{j} d_{i-1} & \forall 1\le j+1 < i \le n;\label{eqn:simpl4}
\end{align}
\item \emph{a weakly simplicial object} if moreover
\begin{align}
d_{i} s_{i} &= d_{i+1} s_{i} & \forall 1\le i \le n;\label{eqn:simpl5}
\end{align}
\item \emph{a simplicial object} if moreover
\begin{align}
d_{i} s_{i} &= \Id_{C_n}  & \forall 1\le i \le n;\label{eqn:simpl6}
\end{align}
\item \emph{a pre-bisimplicial object} if \eqref{eqn:simpl1} holds for the $d_i$'s, $d'_i$'s and their mixture:
\begin{align}
d_{i} d'_{j} &= d'_{j-1} d_{i} & \forall 1\le i < j \le n,\label{eqn:simpl1'}\\
d'_{i} d_{j} &= d_{j-1} d'_{i} & \forall 1\le i < j \le n;\label{eqn:simpl1''}
\end{align}
\item \emph{a (weak / very weak) bisimplicial object} if it is pre-bisimplicial, with both $(C_n,d_{n;i},s_{n;i})$ and $(C_n,d'_{n;i},s_{n;i})$ giving (weak / very weak) simplicial structures.
\end{itemize}
The omitted subscripts $n,n \pm 1$ are those which guarantee that the source of all the above mentioned morphisms is $C_n.$ The $d_i$'s and $s_i$'s are called \emph{face} (resp. \emph{degeneracy}) morphisms.
\end{definition}

Simplicial objects are interesting because of the following properties (see \cite{Cyclic} for most proofs):
\begin{enumerate}
\item if $\C$ is preadditive, then for any presimplicial object $(C_n,d_{n;i})$ in $\C,$
$$\partial_n:=\sum_{i=1}^n (-1)^{i-1} d_{n,i}$$
is a differential (called \emph{total differential}) for the family $(C_n)_{n\ge 0}$, and for any pre-bisimplicial object $(C_n,d_{n;i},d'_{n;i})$ in $\C,$ the differentials $\partial_n$ and 
$$\partial'_n:=\sum_{i=1}^n (-1)^{i-1} d'_{n,i}$$
give a bidifferential structure;
\item in $\C=\kVect,$ for any weakly simplicial object $(C_n,d_{n;i},s_{n;i})$ the complex $(C_n,\partial_n)$ contains \emph{the degenerate subcomplex}
$$D_n:=\sum_{i=1}^{n-1} s_{n-1;i}(C_{n-1});$$
\item if our object turns out to be simplicial, then the degenerate subcomplex is acyclic, hence $C_*$ is quasi-isomorphic to the \emph{normalized compex} $C_*/D_*.$
\end{enumerate}

Proposition \ref{thm:BraidedHom} contains a source of pre-bisimplicial objects. Let us add some notions useful for constructing degeneracies.

\begin{definition}
Take a monoidal category $\C.$ A weakly braided object $(V,\sigma=\sigma_V)$ endowed with a coassociative comultiplication $\Delta:V\rightarrow V\otimes V$  is called \emph{a braided coalgebra} if the structures are compatible:
\begin{align}
\Delta_2\circ\sigma &= \sigma_1 \circ \sigma_2 \circ\Delta_1:V^{\otimes 2} \rightarrow V^{\otimes 3},\label{eqn:BrCoalg}\\
\Delta_1\circ\sigma &= \sigma_2 \circ \sigma_1 \circ\Delta_2:V^{\otimes 2} \rightarrow V^{\otimes 3}.\label{eqn:BrCoalg'}
\end{align}
One talks about \emph{semi-braided coalgebras} if only \eqref{eqn:BrCoalg} holds.
\end{definition}

The compatibilities are graphically depicted as 
  \begin{center}
\begin{tikzpicture}[scale=0.4]
 \draw[rounded corners](0,0) -- (2,2)-- (2,2.5);
 \draw (1,1) -- (1,2.5);
 \draw (1,0) -- (0.65,0.35);
 \draw[rounded corners](0.35,0.65) -- (0,1)-- (0,2.5);
 \fill[teal] (1,1) circle (0.2);
 \node at (3,1) {$=$};
\end{tikzpicture}
\begin{tikzpicture}[scale=0.4]
 \draw[rounded corners](0,-0.5) -- (0,0)-- (2,2);
 \draw[rounded corners](0,-0.5) -- (0,1)-- (1,2);
 \draw (1,-0.5) -- (1,0.7);
 \draw (0.8,1.2) -- (0.6,1.4);
 \draw (0,2) -- (0.35,1.65);
 \fill[teal] (0.05,0) circle (0.2);
\end{tikzpicture}
\begin{tikzpicture}[scale=0.4]
 \node at (-3,1) {};
 \draw[rounded corners](1,0) -- (2,1)-- (2,2.5);
 \draw (1,1) -- (1,2.5);
 \draw (2,0) -- (1.65,0.35);
 \draw[rounded corners](1.35,0.65) -- (0,2)-- (0,2.5);
 \fill[teal] (1,1) circle (0.2);
 \node at (3,1) {$=$};
\end{tikzpicture}
\begin{tikzpicture}[scale=0.4]
 \draw[rounded corners](0,-0.5) -- (0,0)-- (2,2);
 \draw[rounded corners](2,-0.5) -- (2,1)-- (1.65,1.35);
 \draw[rounded corners](2,-0.5) -- (2,0)-- (1.15,0.85);
 \draw[rounded corners](0,2)-- (0.85,1.15);
 \draw (1,2) -- (1.35,1.65);
 \fill[teal] (2,0) circle (0.2);
\end{tikzpicture}
 \end{center}

\begin{theorem}\label{thm:BraidedSimplHom}
\begin{enumerate}
\item 
Let $(\C,\otimes,\II)$ be a preadditive monoidal category. For any weakly braided object $(V,\sigma_V)$ with an {upper cut} $\epsilon,$ there is a pre-bisimplicial structure on $C_n:=V^n$ given by
$$({^\epsilon}\! d)_{n,i} := \epsilon_1 \circ (\sigma_{V^{i-1},V}\otimes \Id_V^{n-i}),$$
$$(d{^\epsilon})_{n,i} := \epsilon_n \circ (\Id_V^{i-1}\otimes \sigma_{V,V^{n-i}}).$$
The corresponding total differentials are precisely $({^\epsilon}\! d)_{n}$ and $(d{^\epsilon})_{n}$ respectively.
\item If a comultiplication $\Delta$ endows $(V,\sigma_V)$ with a braided coalgebra structure, then the morphisms
$$s_{n,i} := \Delta_i$$
complete the preceding structure into a very weakly bisimplicial one.
\item If a comultiplication $\Delta$ endows $(V,\sigma_V)$ with a semi-braided coalgebra structure, then $(C_n,({^\epsilon}\! d)_{n,i},s_{n,i})$ is a very weakly simplicial object.
\item If $\Delta$ is moreover $\sigma$-cocommutative, i.e.
$$\sigma\circ\Delta = \Delta,$$
then the above structures on $V^n$ are weakly (bi)simplicial.
\end{enumerate}
\end{theorem}

The face and degeneracy maps are graphically depicted as 

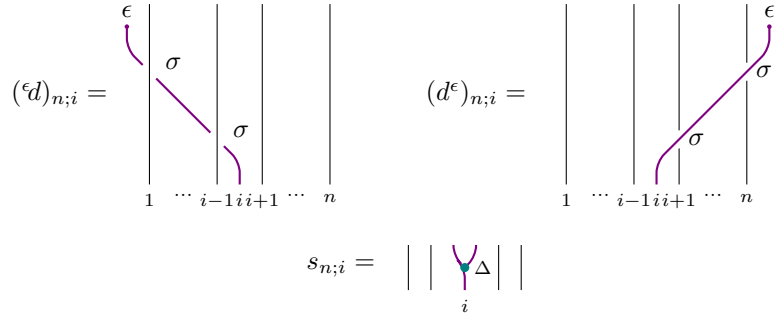
\begin{center}
\begin{tikzpicture}[scale=0.3]
 \draw (0,1) -- (0,-7);
 \node at (-4,-3) {$({^\epsilon}\! d)_{n;i}=$};
 \draw[thick,violet,rounded corners] (-1,0) -- (-1,-1) -- (-0.3,-1.7);
 \draw[thick,violet] (0.3,-2.3) -- (2.7,-4.7);
 \draw[thick,violet,rounded corners] (3.3,-5.3) -- (4,-6) -- (4,-7);
 \fill[violet] (-1,0) circle (0.1);
 \node at (-1,0)  [above] {$\epsilon$};
 \draw (3,1) -- (3,-7);
 \draw (5,1) -- (5,-7);
 \draw (8,1) -- (8,-7);
 \node at (0.3,-2.3) [ above right] {$\sigma$};
 \node at (3.3,-5.3) [above right] {$\sigma$};
 \node at (0,-7) [below] {$\scriptstyle 1$};
 \node at (1.5,-7) [below] {$\scriptstyle\ldots$};
 \node at (3,-7) [below] {${\scriptstyle i-1}$}; 
 \node at (4,-7) [below] {$\scriptstyle i$};
 \node at (5,-7) [below] {$\scriptstyle i+1$};
 \node at (6.5,-7) [below] {$\scriptstyle \ldots$};
 \node at (8,-7) [below] {$\scriptstyle{n}$}; 
\end{tikzpicture}
\begin{tikzpicture}[scale=0.3]
 \draw (0,1) -- (0,-7);
 \node at (-9,-3) {};
 \node at (-4,-3) {$(d^\epsilon)_{n;i}=$};
 \draw[thick,violet,rounded corners] (9,0) -- (9,-1) -- (4,-6) -- (4,-7);
 \fill[violet] (9,0) circle (0.1);
 \node at (9,0) [above] {$\epsilon$};
 \draw (3,1) -- (3,-7);
 \draw (5,1) -- (5,-4.6);
 \draw (8,1) -- (8,-1.6);
 \draw (5,-5.4) -- (5,-7);
 \draw (8,-2.4) -- (8,-7);
 \node at (5,-5) [right] {$\sigma$};
 \node at (8,-2) [right] {$\sigma$};
 \node at (0,-7) [below] {$\scriptstyle 1$};
 \node at (1.5,-7) [below] {$\scriptstyle \ldots$};
 \node at (3,-7) [below] {$\scriptstyle{i-1}$}; 
 \node at (4,-7) [below] {$\scriptstyle i$};
 \node at (5,-7) [below] {$\scriptstyle i+1$};
 \node at (6.5,-7) [below] {$\scriptstyle\ldots$};
 \node at (8,-7) [below] {$\scriptstyle{n}$}; 
\end{tikzpicture}
\end{center}
\begin{center}
\begin{tikzpicture}[scale=0.3]
 \node at (-7,1) {};
 \node at (-4,1) {$s_{n;i}=$};
 \draw (-1,0) -- (-1,2);
 \draw (0,0) -- (0,2);
 \draw (4,0) -- (4,2);
 \draw (3,0) -- (3,2);
 \draw[thick,violet,rounded corners] (1.5,0) -- (1.5,1) -- (1,1.5) -- (1,2);
 \draw[thick,violet,rounded corners] (1.5,1) -- (2,1.5) -- (2,2);
 \fill[teal] (1.5,1) circle (0.2);
 \node at (1.5,1) [right] {$\scriptstyle\Delta$};
 \node at (1.5,0) [below] {$\scriptstyle i$};
\end{tikzpicture}
\end{center}

\begin{proof}
\begin{enumerate}
\item Use Yang-Baxter equation and the definition of a cut.
\item The assertion for the $({^\epsilon}\! d)_{n;i}$'s follows from the first part \eqref{eqn:BrCoalg} of the compatibility condition for the braiding and the comultiplication, while the assertion for the $(d^\epsilon)_{n;i}$'s follows from the second part \eqref{eqn:BrCoalg'}.
\end{enumerate}
\end{proof}

\begin{remark}
When checking the axioms of different types of simplicial structures in this theorem, one can get rid of tiresome index chasing by reasoning in terms of strands. For example, pulling a strand to the left commutes with applying the branching $\Delta$ to another strand if a strand can pass under a branching.
\end{remark}

\subsection{Categorical spindle as a cocommutative braided coalgebra}
Now, using theorem \ref{thm:BraidedSimplHom}, we look for different kinds of simplicial structures for a shelf $(V,\Delta,\lhd)$ in a symmetric category $(\C,\otimes,\II,c),$ generalizing the results established for usual SD structures (cf. \cite{Prz1}).

Endow $V$ with the braiding from proposition \ref{thm:shelf_gen}. An important ingredient missing for applying the above theorem is an upper cut. Here are nice candidates:
\begin{definition}
A \emph{character} for a shelf $(V,\Delta,\lhd)$ in $\C$ is a morphism $\epsilon:V \rightarrow \II$ compatible with $\Delta$ and $\lhd$:
\begin{align*}
(\epsilon\otimes\epsilon)\circ \Delta = \epsilon: &\qquad V\longrightarrow \II,\\
\epsilon\circ\lhd = \epsilon\otimes\epsilon: &\qquad V\otimes V\longrightarrow \II.
\end{align*}
\end{definition}

One easily checks
\begin{lemma}
A character for a shelf $(V,\Delta,\lhd)$ in $\C$ is an upper cut for the weakly braided object $(V,\sigma_{V,\Delta,\lhd}).$ 
\end{lemma}

Further, let us study the compatibility of the braiding $\sigma_{V,\Delta,\lhd}$  with the comultiplication $\Delta,$ in the braided coalgebra sense. For this, a categorical version of the notion of spindle is necessary. Recall notations \eqref{eqn:phi_i} and \eqref{eqn:phi^i}.

\begin{definition}
A shelf $(V,\Delta,\lhd)$ in a symmetric category $(\C,\otimes,\II,c)$ is called a \emph{spindle} if $\Delta$ is \emph{left-cocommutative}:
$$c_1\circ\Delta^2 =\Delta^2$$
and $\lhd$ is \emph{$\Delta$-idempotent}:
$$\lhd\circ\Delta = \Id_V.$$
\end{definition}
It is the graphical form \begin{tikzpicture}[scale=0.3]
 \draw (0,0) -- (0,0.5);
 \draw [rounded corners] (0,0.5) -- (0.5,1) -- (0,1.5);
 \draw [rounded corners] (0,0.5) -- (-0.5,1) -- (0,1.5);
 \draw (0,2) -- (0,1.5);
 \node at (1.2,1) {$\scriptstyle{=}$};
 \draw (2,2) -- (2,0);
 \node at (2.5,1) {};
\end{tikzpicture} of the last condition that explains the term. It was coined by Alissa Crans in \cite{Crans}, as well as the term ``shelf".

Note that left cocommutativity is stronger than the weak one and weaker as the usual one.

\begin{proposition}\label{thm:SpindleBrCoalg}
Take a shelf $(V,\Delta,\lhd)$ in a symmetric category $(\C,\otimes,\II,c).$ The data $(V,\sigma=\sigma_{V,\Delta,\lhd},\Delta)$ defines a semi-braided coalgebra, $\sigma$-cocommutative if $V$ is a spindle.
\end{proposition}

\begin{proof}
Compatibility relation \eqref{eqn:BrCoalg} follows from the bialgebra-type compatibility between $\Delta$ and $\lhd,$ and $\sigma$-cocommutativity from the two properties defining spindles.
\end{proof}

The additional conditions in the definition of a spindle turn out not to be too restrictive:
\begin{proposition}
The following GSD structures are spindles:
\begin{enumerate}
\item usual spindles in $\Set$;
\item UAAs (for which $\Delta$-idempotence is equivalent to $\nu$ being a left unit);
\item ULAs;
\item cocommutative Hopf algebras.
\end{enumerate}
\end{proposition}

\medskip
Plugging proposition \ref{thm:SpindleBrCoalg} into theorem \ref{thm:BraidedSimplHom}, one gets:

\begin{theorem}\label{thm:GSDhom}
Let $\epsilon$ be a character for a shelf $(V,\Delta,\lhd)$ in a symmetric preadditive category $(\C,\otimes,\II,c).$ Morphisms
\begin{align*}
({^\epsilon}\! d)_{n;i}& := (( \epsilon\otimes \lhd^{\otimes(i-1)})\circ\theta_{(i)} \circ (\Delta^{i-1})_i) \otimes \Id_V^{n-i} :V^{n}\rightarrow V^{n-1},\\
(d{^\epsilon})_{n;i} &:=   \Id_V^{i-1}\otimes\epsilon \otimes\chi^{\otimes(n-i)} :V^{n}\rightarrow V^{n-1},
\end{align*}
define then a pre-bisimplicial structure on $C_n:=V^n,$ where
$$\chi = \epsilon_2 \circ \Delta:V\longrightarrow V,$$
and $\theta_{(i)} = \bigl( \begin{smallmatrix}
 1 & 2 & \cdots & i-1    & i & i+1& \cdots & 2i-1\\ 
 2 & 4 & \cdots & 2(i-1) & 1 & 3  & \cdots & 2i-1
\end{smallmatrix} \bigr) \in S_{2i-1}$ acts on the powers of $V$ via the symmetric braiding $c.$

Further, $(C_n,({^\epsilon}\! d)_{n,i},s_{n,i} := \Delta_i)$ is a very weakly simplicial object, becoming weakly simplicial if $V$ is a spindle.
\end{theorem}

As a consequence, any linear combination of total differentials ${^\epsilon}\! d$ and $d{^\epsilon}$ defines a differential for $V,$ and thus a homology theory if the category $\C$ is sufficiently nice.

For usual shelves, UAAs and ULAs, these total differentials coincide with those from \cite{Lebed1} (where they were obtained directly, without the formalism of GSD structures) and thus recover, for suitable choice of characters, many known homologies. Moreover, for usual spindles and unital associative algebras, the weakly simplicial structures $(V^n,({^\epsilon}\! d)_{n;i},\Delta_i)$  are precisely the familiar ones.

\begin{remark}
In fact in the settings of the theorem one has a {(very) weakly \textbf{bi}simplicial structure}: although the second braided coalgebra condition \eqref{eqn:BrCoalg'} does not hold in general, one checks directly the relations between the $(d{^\epsilon})_{n;i} $'s and the $s_{n,i}$'s, using the weak cocommutativity of $\Delta$ and its compatibility with the character $\epsilon.$

Note also that if the character $\epsilon$ is a counit for $\Delta,$ then the structure $(V^n,(d^\epsilon)_{n;i},\Delta_i)$ is simplicial.
\end{remark}

\begin{remark}
Observe that the face and degeneracy morphisms from theorem \ref{thm:GSDhom} are of very different nature. The author does not see how to make the definition symmetric.
\end{remark}

\bibliographystyle{plain}
\bibliography{biblio}

\def\cprime{$'$} \def\cprime{$'$}
\begin{thebibliography}{10}

\bibitem{Baez}
John~C. Baez.
\newblock Hochschild homology in a braided tensor category.
\newblock {\em Trans. Amer. Math. Soc.}, 344(2):885--906, 1994.

\bibitem{BartFenn}
Andrew Bartholomew and Roger Fenn.
\newblock Quaternionic invariants of virtual knots and links.
\newblock {\em J. Knot Theory Ramifications}, 17(2):231--251, 2008.

\bibitem{Birman}
Joan~S. Birman.
\newblock {\em Braids, links, and mapping class groups}.
\newblock Princeton University Press, Princeton, N.J., 1974.
\newblock Annals of Mathematics Studies, No. 82.

\bibitem{CatSelfDistr}
J.~Scott Carter, Alissa~S. Crans, Mohamed Elhamdadi, and Masahico Saito.
\newblock Cohomology of categorical self-distributivity.
\newblock {\em J. Homotopy Relat. Struct.}, 3(1):13--63, 2008.

\bibitem{HomologyYB}
J.~Scott Carter, Mohamed Elhamdadi, and Masahico Saito.
\newblock Homology theory for the set-theoretic {Y}ang-{B}axter equation and
  knot invariants from generalizations of quandles.
\newblock {\em Fund. Math.}, 184:31--54, 2004.

\bibitem{Crans}
Alissa~Susan Crans.
\newblock {\em Lie 2-algebras}.
\newblock ProQuest LLC, Ann Arbor, MI, 2004.
\newblock Thesis (Ph.D.)--University of California, Riverside.

\bibitem{Dehornoy}
Patrick Dehornoy, Ivan Dynnikov, Dale Rolfsen, and Bert Wiest.
\newblock {\em Ordering braids}, volume 148 of {\em Mathematical Surveys and
  Monographs}.
\newblock American Mathematical Society, Providence, RI, 2008.

\bibitem{LeiColor}
A.~S. Dzhumadil{\cprime}daev.
\newblock Cohomologies of colour {L}eibniz algebras: pre-simplicial approach.
\newblock In {\em Lie theory and its applications in physics, {III}
  ({C}lausthal, 1999)}, pages 124--136. World Sci. Publ., River Edge, NJ, 2000.

\bibitem{SymmCat}
Samuel Eilenberg and G.~Max Kelly.
\newblock Closed categories.
\newblock In {\em Proc. {C}onf. {C}ategorical {A}lgebra ({L}a {J}olla,
  {C}alif., 1965)}, pages 421--562. Springer, New York, 1966.

\bibitem{FRR}
Roger Fenn, Rich{\'a}rd Rim{\'a}nyi, and Colin Rourke.
\newblock The braid-permutation group.
\newblock {\em Topology}, 36(1):123--135, 1997.

\bibitem{FR}
Roger Fenn and Colin Rourke.
\newblock Racks and links in codimension two.
\newblock {\em J. Knot Theory Ramifications}, 1(4):343--406, 1992.

\bibitem{Goyvaerts}
I.~{Goyvaerts} and J.~{Vercruysse}.
\newblock {A Note on the categorification of Lie algebras}.
\newblock {\em ArXiv e-prints}, February 2012.

\bibitem{BraidedCat}
Andr{\'e} Joyal and Ross Street.
\newblock Braided tensor categories.
\newblock {\em Adv. Math.}, 102(1):20--78, 1993.

\bibitem{Joyce}
David Joyce.
\newblock A classifying invariant of knots, the knot quandle.
\newblock {\em J. Pure Appl. Algebra}, 23(1):37--65, 1982.

\bibitem{Kamada2}
Naoko Kamada and Seiichi Kamada.
\newblock Abstract link diagrams and virtual knots.
\newblock {\em J. Knot Theory Ramifications.}, 9:93--106, 2000.

\bibitem{Kamada}
Seiichi Kamada.
\newblock Knot invariants derived from quandles and racks.
\newblock In {\em Invariants of knots and 3-manifolds ({K}yoto, 2001)},
  volume~4 of {\em Geom. Topol. Monogr.}, pages 103--117 (electronic). Geom.
  Topol. Publ., Coventry, 2002.

\bibitem{KamadaVirtual}
Seiichi Kamada.
\newblock Braid presentation of virtual knots and welded knots.
\newblock {\em Osaka J. Math.}, 44(2):441--458, 2007.

\bibitem{KasselDehornoy}
Christian Kassel.
\newblock L'ordre de {D}ehornoy sur les tresses.
\newblock {\em Ast\'erisque}, (276):7--28, 2002.
\newblock S{\'e}minaire Bourbaki, Vol. 1999/2000.

\bibitem{Braids}
Christian Kassel and Vladimir Turaev.
\newblock {\em Braid groups}, volume 247 of {\em Graduate Texts in
  Mathematics}.
\newblock Springer, New York, 2008.
\newblock With the graphical assistance of Olivier Dodane.

\bibitem{KauffmanVirtual}
Louis~H. Kauffman.
\newblock Virtual knot theory.
\newblock {\em European J. Combin.}, 20(7):663--690, 1999.

\bibitem{KauffLambr}
Louis~H. Kauffman and Sofia Lambropoulou.
\newblock A categorical structure for the virtual braid group.
\newblock {\em Comm. Algebra}, 39(12):4679--4704, 2011.

\bibitem{KaufRad}
Louis~H. Kauffman and David Radford.
\newblock Bi-oriented quantum algebras, and a generalized {A}lexander
  polynomial for virtual links.
\newblock In {\em Diagrammatic morphisms and applications ({S}an {F}rancisco,
  {CA}, 2000)}, volume 318 of {\em Contemp. Math.}, pages 113--140. Amer. Math.
  Soc., Providence, RI, 2003.

\bibitem{Lebed1}
Victoria Lebed.
\newblock {Homologies of Algebraic Structures via Braidings and Quantum
  Shuffles}.
\newblock {\em ArXiv e-prints}, April 2012.

\bibitem{LeiSuper}
Dong Liu and Naihong Hu.
\newblock Leibniz superalgebras and central extensions.
\newblock {\em J. Algebra Appl.}, 5(6):765--780, 2006.

\bibitem{Cyclic}
Jean-Louis Loday.
\newblock {\em Cyclic homology}, volume 301 of {\em Grundlehren der
  Mathematischen Wissenschaften [Fundamental Principles of Mathematical
  Sciences]}.
\newblock Springer-Verlag, Berlin, 1992.
\newblock Appendix E by Mar{\'{\i}}a O. Ronco.

\bibitem{Cat}
Saunders MacLane.
\newblock {\em Categories for the working mathematician}.
\newblock Springer-Verlag, New York, 1971.
\newblock Graduate Texts in Mathematics, Vol. 5.

\bibitem{Majid2}
Shahn Majid.
\newblock Algebras and {H}opf algebras in braided categories.
\newblock In {\em Advances in {H}opf algebras ({C}hicago, {IL}, 1992)}, volume
  158 of {\em Lecture Notes in Pure and Appl. Math.}, pages 55--105. Dekker,
  New York, 1994.

\bibitem{Manturov}
V.~O. Manturov.
\newblock On invariants of virtual links.
\newblock {\em Acta Appl. Math.}, 72(3):295--309, 2002.

\bibitem{ManRecognition}
V.~O. Manturov.
\newblock On the recognition of virtual braids.
\newblock {\em Zap. Nauchn. Sem. S.-Peterburg. Otdel. Mat. Inst. Steklov.
  (POMI)}, 299(Geom. i Topol. 8):267--286, 331--332, 2003.

\bibitem{ManRecognitionEn}
V.~O. Manturov.
\newblock Recognition of virtual braids.
\newblock {\em Journal of Mathematical Sciences}, 131(1):5409--54192, 2005.

\bibitem{Matveev}
S.~V. Matveev.
\newblock Distributive groupoids in knot theory.
\newblock {\em Mat. Sb. (N.S.)}, 119(161)(1):78--88, 160, 1982.

\bibitem{Nelson}
Sam Nelson.
\newblock Virtual knots.
\newblock \url{http://www.esotericka.org/cmc/vknots.html}.

\bibitem{Prz1}
J{\'o}zef~H. Przytycki.
\newblock Distributivity versus associativity in the homology theory of
  algebraic structures.
\newblock {\em Demonstratio Math.}, 44(4):823--869, 2011.

\bibitem{SW}
Daniel~S. Silver and Susan~G. Williams.
\newblock A generalized {B}urau representation for string links.
\newblock {\em Pacific J. Math.}, 197(1):241--255, 2001.

\bibitem{Invariants}
V.~G. Turaev.
\newblock {\em Quantum invariants of knots and 3-manifolds}, volume~18 of {\em
  de Gruyter Studies in Mathematics}.
\newblock Walter de Gruyter \& Co., Berlin, 1994.

\bibitem{Ver}
Vladimir~V. Vershinin.
\newblock On homology of virtual braids and {B}urau representation.
\newblock {\em J. Knot Theory Ramifications}, 10(5):795--812, 2001.
\newblock Knots in Hellas '98, Vol. 3 (Delphi).

\bibitem{Wor}
S.~L. Woronowicz.
\newblock Solutions of the braid equation related to a {H}opf algebra.
\newblock {\em Lett. Math. Phys.}, 23(2):143--145, 1991.

\end{thebibliography}

\end{document}